\documentclass[10pt,a4paper]{article}

\usepackage{amssymb}
\usepackage{amsmath,lipsum}
\usepackage{graphicx}
\usepackage[T1]{fontenc} 
\usepackage[english]{babel}
\usepackage{amsthm,stmaryrd}
\usepackage{tabulary}
\usepackage{booktabs}
\usepackage{mathtools}
\usepackage{tikz}
\usepackage{comment}
\usepackage{hyperref}
\usepackage{shuffle}
\usepackage{float}
\usepackage{tikz-cd}

\usepackage{difftrees}
\tikzset{
  ncone/.pic={
	\draw (0,0)--(0,0.2);
  }
}

\tikzset{
  nctwo/.pic={
    \draw (0,0)--(0,0.2);
	\draw (0.1,0)--(0.1,0.2);
  }
}

\tikzset{
  nctwoW/.pic={
    \draw (0,0.2)--(0,0)--(0.1,0)--(0.1,0.2);
  }
}

\tikzset{
  nctwoWW/.pic={
    \draw (0,0.2)--(0,0)--(0.2,0)--(0.2,0.2);
  }
}

\tikzset{
  ncthreeWW/.pic={
    \draw (0,0.2)--(0,0)--(0.3,0)--(0.3,0.2);
	\draw (0.2,0)--(0.2,0.2);
  }
}

\tikzset{
  ncthree/.pic={
    \draw (0,0)--(0,0.2);
	\draw (0.1,0)--(0.1,0.2);
	\draw (0.2,0)--(0.2,0.2);
  }
}

\tikzset{
  ncthreeW/.pic={
    \draw (0,0.2)--(0,0)--(0.2,0)--(0.2,0.2);
	\draw (0.1,0)--(0.1,0.2);
  }
}

\tikzset{
  ncfour/.pic={
    \draw (0,0)--(0,0.2);
	\draw (0.1,0)--(0.1,0.2);
	\draw (0.2,0)--(0.2,0.2);
	\draw (0.3,0)--(0.3,0.2);
  }
}
\tikzset{
  ncfive/.pic={
    \draw (0,0)--(0,0.2);
	\draw (0.1,0)--(0.1,0.2);
	\draw (0.2,0)--(0.2,0.2);
	\draw (0.3,0)--(0.3,0.2);
	\draw (0.4,0)--(0.4,0.2);
  }
}

\tikzset{
  ncfourW/.pic={
    \draw (0,0.2)--(0,0)--(0.3,0)--(0.3,0.2);
	\draw (0.1,0)--(0.1,0.2);
	\draw (0.2,0)--(0.2,0.2);
  }
}
\tikzset{
  ncfiveW/.pic={
    \draw (0,0.2)--(0,0)--(0.4,0)--(0.4,0.2);
	\draw (0.1,0)--(0.1,0.2);
	\draw (0.2,0)--(0.2,0.2);
	\draw (0.3,0)--(0.3,0.2);
  }
}

\tikzset{
  nconeinsidetwoWW/.pic={
    \path (0,0) pic {nctwoWW}; \path (0.1,0.1) pic {ncone};
  }
}

\tikzset{
  nconeinsidethreeWW/.pic={
    \path (0,0) pic {ncthreeWW}; \path (0.1,0.1) pic {ncone};
  }
}
\tikzset{
  nconeinsidethreerightWW/.pic={
    \draw (0,0.2)--(0,0)--(0.3,0)--(0.3,0.2);
    \draw (0.1,0)--(0.1,0.2); \draw (0.2,0.1)--(0.2,0.3);
  }
}
\tikzset{
  nconeinsidethreeleftWW/.pic={
    \draw (0,0.2)--(0,0)--(0.3,0)--(0.3,0.2);
    \draw (0.2,0)--(0.2,0.2); \draw (0.1,0.1)--(0.1,0.3);
  }
}
\tikzset{
  nctwoWWW/.pic={
    \draw (0,0.2)--(0,0)--(0.3,0)--(0.3,0.2);
  }
}

\tikzset{
  nconeoneinsidetwoWW/.pic={
	\path (0,0) pic {ncone};
    \path (0.1,0) pic {nctwoWW}; 
	\path (0.2,0.1) pic {ncone};
  }
}

\usepackage{forest}
\forestset{
  decor/.style = {
    label/.expanded = {[inner sep = 0.2ex, font=\unexpanded{\tiny}]right:{$#1$}}
  },
  root/.style = {minimum size = 0.1ex},
  decorated/.style = {
    for tree = {
      circle, fill, inner sep = 0.3ex, minimum size = 1.ex,
      grow' = south, l = 0, l sep = 1.2ex, s sep = 0.7em,
      fit = tight, parent anchor = center, child anchor = center,
      delay = {decor/.option = content, content =}
    }
  },
  default preamble = {decorated, root},
  begin draw/.code={\begin{tikzpicture}[baseline={([yshift=-0.5ex]current bounding box.center)}]},
}

\newtheorem{theorem}{Theorem}[section]
\newtheorem{corollary}{Corollary}[theorem]
\newtheorem{lemma}[theorem]{Lemma}
\newtheorem{proposition}[theorem]{Proposition}
\newtheorem{definition}{Definition}
\newtheorem{example}[theorem]{Example}
\newtheorem{remark}{Remark}
\newtheorem{notation}{Notation}

\DeclareMathOperator{\graft}{\triangleright}

\title{An operadic approach to substitution in Lie--Butcher series}

\author{Ludwig Rahm\footnote{Department of Mathematical Sciences, Norwegian University of Science and Technology (NTNU), 7491 Trondheim, Norway. \texttt{ludwig.rahm@ntnu.no}.}}

\begin{document}

\maketitle

\begin{abstract}
The paper follows an operadic approach to provide a bialgebraic description of substitution for Lie--Butcher series. We first show how the well-known bialgebraic description for substitution in Butcher's $B$-series can be obtained from the pre-Lie operad. We then apply the same construction to the post-Lie operad to arrive at a bialgebra $\mathcal{Q}$. By considering a module over the post-Lie operad, we get a cointeraction between $\mathcal{Q}$ and the Hopf algebra $\mathcal{H}_N$ that describes composition for Lie--Butcher series. We use this coaction to describe substitution for Lie--Butcher series.
\end{abstract}

\tableofcontents

\section{Introduction}
\label{sec:intro}

Many numerical integration methods for differential equations defined on Euclidean spaces have been understood and studied through the formalism of $B$-series introduced by John Butcher \cite{Butcher1972,Butcher2008,CalvoSanz-Serna1994,HairerWanner1974}. Integration methods that can be formulated by $B$-series (so-called $B$-series methods) have been studied with an emphasis on algebraic structures defined on non-planar rooted trees \cite{MclachlanModinMunthe-KaasVerdier2017}. Informally speaking, a $B$-series is a Taylor series with terms indexed by non-planar rooted trees, together with an algebra morphism that maps the trees to a vector field and its derivatives. \\
In his study of Runge--Kutta methods on Lie groups \cite{Munthe-Kaas1995,Munthe-Kaas1998}, Munthe-Kaas defined the notion of Lie--Butcher series (LB-series). They play a role on homogeneous spaces, similar to that of Butcher's $B$-series on Euclidean spaces. The study of LB-series methods emphasises algebraic structures defined on planar rooted trees \cite{CurryEbrahimi-FardMunthe-Kaas2017,LundervoldMunthe-Kaas2011,Munthe-KaasWright2008,Munthe-KaasLundervold2012}. A common theme for these algebraic structures is that they specialise to the corresponding structures for $B$-series methods, when planarity for the trees is removed. A non-planar tree is a tree seen as a graph, a planar tree is a tree endowed with an embedding into the plane. The free pre-Lie algebra is one of the essential structures on non-planar rooted trees. The planar generalization of pre-Lie is the free post-Lie algebra, which is defined over formal Lie brackets of planar rooted trees.  \\
It is of particular interest for the present paper to consider the notions of composition and substitution of LB-series. Connes and Kreimer \cite{ConnesKreimer1998} described the Hopf algebra that governs composition of $B$-series, by using the notion of admissible edge cuts in non-planar rooted trees. The main idea of $B$-series composition is that the flow of a differential equation can be described by a $B$-series and one aims to study the composition of flows as a composition of $B$-series. Munthe-Kaas and Wright \cite{Munthe-KaasWright2008} generalised this to the notion of admissible left edge cuts in planar rooted forests with the goal of describing the Hopf algebra that governs composition of LB-series. Calaque, Ebrahimi-Fard and Manchon \cite{CalaqueEbrahimi-FardManchon2011} described the so-called extraction-contraction bialgebra $\mathcal{H}$ that governs substitution of $B$-series, by using edge contractions in non-planar rooted trees. Calaque et al.~furthermore described a cointeraction of their bialgebra with the Hopf algebra of Connes and Kreimer. The idea behind $B$-series substitution is that a $B$-series, while being a sum over a vector field and its derivatives, can itself also describe a vector field. In this case, it makes sense to consider a $B$-series that has another $B$-series as its vector field, which we call substitution. A recursive formula for substitution in LB-series has been given by Lundervold and Munthe-Kaas \cite{LundervoldMunthe-Kaas2011}. The algebraic picture of a bialgebra cointeracting with the Hopf algebra of Munthe-Kaas and Wright is however not present. Subsitution was also considered in \cite{FloystadMunthe-Kaas}, where Algebro-Geometric methods were used to show that there is a bialgebraic description. Their construction was however not made explicit.

The paper at hand applies operadic methods to obtain a bialgebra of cosubstitution for LB-series. We use a construction by Foissy \cite{Foissy2017}, dualizing operadic composition into a coproduct. We then show that applying this construction to the pre-Lie operad results in the bialgebra $\mathcal{H}$ that was used to describe substitution in $B$-series \cite{CalaqueEbrahimi-FardManchon2011}. The pre-Lie operad is defined by replacing vertices in a non-planar rooted tree by rooted trees, which is also how we think about concrete $B$-series substitution. This perspective motivates us to look at the post-Lie operad, which can be described by replacing vertices in planar rooted trees by Lie polynomials of planar rooted trees. Applying Foissy's construction to the post-Lie operad gives us a bialgebra $\mathcal{Q}$ defined over $S(Lie(\mathcal{PT}))$, the symmetric algebra of Lie polynomials in planar rooted trees. Using the embedding into the space spanned by ordered forests $\mathcal{OF}$ of planar rooted trees, $Lie(\mathcal{PT}) \subset U(Lie(\mathcal{PT}))=\mathcal{OF}$, we endow $\mathcal{OF}$ with a module structure over the post-Lie operad, given by replacing vertices by Lie polynomials. As a matter of fact, replacing vertices by Lie polynomials is how we think about concrete $LB$-series substitution. By dualizing in the same way as in Foissy's construction, the module structure dualizes to a coaction. The latter describes a cointeraction between the bialgebra $\mathcal{Q}$ and the Hopf algebra $\mathcal{H}_N$ of Munthe-Kaas--Wright \cite{LundervoldMunthe-Kaas2011}. We then show how substitution in $LB$-series can be described and computed using this coaction.

\medskip

The structure of the paper is as follow: In section \ref{section::1}, we summarize the definitions and results that the present paper builds upon. In section \ref{section::2}, we construct an operad over non-planar trees. We then prove that our construction provides an alternative description of the pre-Lie operad defined by Chapoton and Livernet in \cite{ChapotonLivernet2001}. The section is concluded by proving a duality between the pre-Lie operad and the coproduct $\Delta_{\mathcal{H}}$ that is used to describe $B$-series substitution. In section \ref{sec:planarSubsOperad}, we construct an operad of Lie brackets over planar trees in a way that is analogous to the operadic construction from section \ref{section::2}. In section \ref{section::3}, we extend the operadic composition from section \ref{sec:planarSubsOperad} to let Lie brackets of planar rooted trees act on forests. We then dualize this to a coaction. In section \ref{section::5}, we prove that the coaction can be used to describe substitution in LB-series. In section \ref{section::4}, we provide a combinatorial picture.

\section{Preliminaries} 
\label{section::1}

We recall some definitions and fundamental results. All algebraic structures are assumed to be defined over some fixed field $\mathbb{K}$ of characteristic zero.

\subsection{Trees and forests} 
\label{ssec:TreesForests}

A non-planar rooted tree is a directed graph with a distinguished vertex, called the root, such that every vertex except the root has exactly one incoming edge. The root has no incoming edges. Vertices without outgoing edges are called leafs. A planar rooted tree is a rooted tree endowed with an embedding into the plane. We will draw trees with the root at the top and edges oriented away from the root. Consider for example the two trees
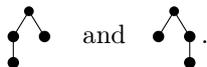
\begin{align*}
\Forest{[[][[]]]} \quad \text{and} \quad  \Forest{[[[]][]]}.
\end{align*}
They are isomorphic as graphs and hence equal as non-planar rooted trees. However, as the embeddings into the plane are different, they are not equal as planar rooted trees. An unordered sequence of non-planar rooted trees is called a forest. An ordered sequence of planar rooted trees is called an ordered forest. We denote the vector space spanned by all non-planar rooted trees by $\mathcal{T}$, the vector space spanned by all planar rooted trees by $\mathcal{PT}$, the vector space spanned by all forests by $\mathcal{F}$ and the vector space spanned by all ordered forests by $\mathcal{OF}$. The empty forest is denoted $\emptyset$.\\
We introduce the grafting operator $\curvearrowright : \mathcal{T} \otimes \mathcal{T} \to \mathcal{T}$ by defining $\tau_1 \curvearrowright \tau_2$ to be the sum over all ways of adding an edge from some vertex in $\tau_2$ to the root of $\tau_1$. For example:
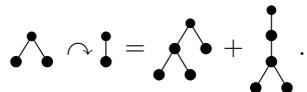
\allowdisplaybreaks
\begin{align*}
\Forest{[[][]]} \curvearrowright \Forest{[[]]}= \Forest{[[][[][]] ]}+\Forest{[[[[][]]]]}.
\end{align*}
Endowing the space $\mathcal{T}$ with the grafting operator $\curvearrowright$ produces a (left) pre-Lie algebra \cite{Burde2006,Cartier11,ChapotonLivernet2001,Livernet2006,Manchon2009,OudomGuin2008}, meaning that the following so-called (left) pre-Lie identity is satisfied
\allowdisplaybreaks
\begin{align*}
\tau_1 \curvearrowright (\tau_2 \curvearrowright \tau_3) - (\tau_1\curvearrowright \tau_2) \curvearrowright \tau_3 - \tau_2 \curvearrowright (\tau_1 \curvearrowright \tau_3) + (\tau_2 \curvearrowright \tau_1) \curvearrowright \tau_3=0, \; \forall \tau_1,\tau_2,\tau_3 \in \mathcal{T},
\end{align*}
which has an obvious combinatorial interpretation. $(\mathcal{T},\curvearrowright )$ is, in fact, the free pre-Lie algebra on one generator. The pre-Lie identity implies that for all $\tau_1,\tau_2 \in \mathcal{T}$, the commutator $\llbracket \tau_1,\tau_2 \rrbracket := \tau_1\curvearrowright \tau_2 - \tau_2 \curvearrowright \tau_1$ satisfies the Jacobi identity.

We furthermore define on planar rooted trees the left grafting operator $\graft: \mathcal{PT} \otimes \mathcal{PT} \to \mathcal{PT}$ by letting $\tau_1 \graft \tau_2$ denote the sum over all ways of adding an edge from any vertex in $\tau_2$ to the root of $\tau_1$ such that the added edge is leftmost on this vertex with respect to the planar embedding. Note that the left grafting operator is magmatic. Extending it to the free Lie algebra, $Lie(\mathcal{PT})$, generated by $\mathcal{PT}$ via the rules
\allowdisplaybreaks
\begin{align*}
\tau_1 \graft [\tau_2,\tau_3]:=&[\tau_1 \graft \tau_2,\tau_3] + [\tau_2,\tau_1 \graft \tau_3], \\
[\tau_1,\tau_2] \graft \tau_3 :=& \tau_1 \graft (\tau_2 \graft \tau_3) - (\tau_1 \graft \tau_2) \graft \tau_3 - \tau_2 \graft (\tau_1 \graft \tau_3) + (\tau_2 \graft \tau_1)\graft \tau_3, \; \forall \tau_1,\tau_2,\tau_3 \in Lie(\mathcal{PT}).
\end{align*}
produces the free post-Lie algebra \cite{CurryEbrahimi-FardMunthe-Kaas2017,Ebrahimi-FardLundervoldMunthe-Kaas2014,LundervoldMunthe-Kaas2011,Munthe-KaasLundervold2012,Silva2018}. For all $\tau_1,\tau_2 \in \mathcal{PT}$, the commutator $\llbracket \tau_1,\tau_2 \rrbracket := \tau_1 \graft \tau_2 - \tau_2 \graft \tau_1 + [\tau_1,\tau_2]$ satisfies the Jacobi identity. Note that in general a post-Lie algebra with a vanishing Lie bracket reduces to a pre-Lie algebra.\\
Defining the commutator, $[\tau_1,\tau_2]=\tau_1\tau_2-\tau_2\tau_1$, one can identify the universal enveloping algebra of the free post-Lie algebra with $\mathcal{OF}$ as vector space. The associative product becomes concatenation of ordered forests. We extend left grafting to $\mathcal{OF}$ by
\allowdisplaybreaks
\begin{align*}
\tau_1 \graft \tau_2\omega_2:=&\ (\tau_1 \graft \tau_2)\omega_2 + \tau_2(\tau_1 \graft \omega_2), \\
\tau_1\omega_1 \graft \omega_2:=&\ \tau_1 \graft (\omega_1 \graft \omega_2) - (\tau_1 \graft \omega_1) \graft \omega_2, \; \forall \tau_1,\tau_2 \in \mathcal{PT}, \; \forall \omega_1,\omega_2 \in \mathcal{OF}.
\end{align*}
The vector space $\mathcal{OF}$ together with left grafting and concatenation is the free $D$-algebra generated by the single-vertex tree \cite{Munthe-KaasWright2008}. See Definition \ref{def:Dalg} in the following subsection.\\
We define the function $B^+: \mathcal{OF} \to \mathcal{PT}$ given by $B^+(\omega)=\omega \graft \bullet$, where $\bullet$ is the single-vertex tree. For example:
\begin{align*}
B^+(\Forest{[[][]]}\Forest{[[]]}\Forest{[[][[][]]]})=\Forest{[ [[][[][]]] [[]] [[][]]]}.
\end{align*}
The inverse of $B^+$, denoted $B^-: \mathcal{PT} \to \mathcal{OF}$, is given by removing the root together with its outgoing edges from the input tree. The operator $\diamond : \mathcal{OF} \times \mathcal{OF} \to \mathcal{OF}$ given by 
\begin{align}
\label{planarGLprod}
	\omega_1 \diamond \omega_2 = B^-(\omega_1 \graft B^+(\omega_2) )
\end{align}	 
is called the --planar-- Grossman--Larson product. For example:
\allowdisplaybreaks
\begin{align*}
\Forest{[[][]]} \diamond \Forest{[]}\Forest{[[]]}=\Forest{[[][]]}\Forest{[]}\Forest{[[]]}+\Forest{[[[][]]]}\Forest{[[]]}+\Forest{[]}\Forest{[[][[][]]]}+\Forest{[]}\Forest{[[[[][]]]]}.
\end{align*}
Let $\omega_1,\omega_2$ be ordered forests, then we can write them as a sequence of trees:
\begin{align*}
\omega_1=&\ \tau_1^1 \cdots \tau_1^n, \\
\omega_2=&\ \tau_2^1 \cdots \tau_2^m.
\end{align*}
We define the shuffle product $\omega_1 \shuffle\ \omega_2$ as the sum of all the ways to concatenate the trees $\tau_1^1,\dots,\tau_1^n,\tau_2^1,\cdots,\tau_2^m$ into a forest such that $\tau_i^j$ is to the left of $\tau_k^{\ell}$ if $i=k$ and $j \leq \ell$. For example:
\begin{align*}
\Forest{[]}\Forest{[[]]}\shuffle \Forest{[[][]]}\Forest{[[[]]]}=\Forest{[]}\Forest{[[]]}\Forest{[[][]]}\Forest{[[[]]]} + \Forest{[]}\Forest{[[][]]}\Forest{[[]]}\Forest{[[[]]]} + \Forest{[[][]]}\Forest{[]}\Forest{[[]]}\Forest{[[[]]]}+\Forest{[[][]]}\Forest{[]}\Forest{[[[]]]}\Forest{[[]]}+\Forest{[]}\Forest{[[][]]}\Forest{[[[]]]}\Forest{[[]]}+\Forest{[[][]]}\Forest{[[[]]]}\Forest{[]}\Forest{[[]]}.
\end{align*}
The empty forest is the unit for the shuffle product, $\omega \shuffle\ \emptyset = \omega = \emptyset \shuffle\ \omega$.
The shuffle coproduct $\Delta_{\shuffle}: \mathcal{OF} \to \mathcal{OF} \otimes \mathcal{OF}$ is defined by $\Delta_{\shuffle}(\omega)$ being the sum of all $\omega_1 \otimes \omega_2$ such that $\omega_1\shuffle\omega_2$ contains the forest $\omega$. For example:
\begin{align*}
\Delta_{\shuffle}(\Forest{[]}\Forest{[[]]}\Forest{[[][]]})
=&\ \Forest{[]}\Forest{[[]]}\Forest{[[][]]} \otimes \emptyset + \Forest{[]}\Forest{[[]]}\otimes \Forest{[[][]]} + \Forest{[]}\Forest{[[][]]}\otimes \Forest{[[]]} + \Forest{[[]]}\Forest{[[][]]} \otimes \Forest{[]}\\
 \qquad +& \Forest{[]} \otimes \Forest{[[]]}\Forest{[[][]]} + \Forest{[[]]} \otimes \Forest{[]}\Forest{[[]]}\Forest{[[][]]} + \Forest{[[][]]} \otimes \Forest{[]}\Forest{[[]]} + \emptyset \otimes \Forest{[]}\Forest{[[]]}\Forest{[[][]]}.
\end{align*}

\subsection{$D$-algebras}
\label{ssec:Dalgebra}

We now recall the definition of a $D$-algebra \cite{LundervoldMunthe-Kaas2011,Munthe-KaasLundervold2012,Munthe-KaasWright2008}.

\begin{definition}
\label{def:Dalg}
Let $(A,\cdot)$ be a unital associative algebra with unit $1$. If $A$ is furthermore equipped with a non-associative product $\graft$, denote by $\mathcal{D}(A)=\{x \in A: x \graft (a \cdot b)= (x \graft a)\cdot b + a \cdot (x \graft b), \; \forall a,b\in A \}$ the set of derivations in $A$. The triple $(A,\cdot,\graft)$ is then called a $D$-algebra if the following identities hold:
\begin{align*}
1 \graft a =&\ a, \\
a \graft x \in&\ \mathcal{D}(A), \\
x \graft (a \graft b) =&\ (x \cdot a)\graft b + (x \graft a) \graft b,
\end{align*}
for $a,b \in A$ and $x \in \mathcal{D}(A)$.
\end{definition}

A map $\phi: A \to A'$ between two $D$-algebras is called a $D$-algebra morphism if:
\begin{align*}
\phi(a\cdot b)=&\ \phi(a) \cdot \phi(b), \\
\phi(a \graft b)=&\ \phi(a) \graft \phi(b), \\
\phi(\mathcal{D}(A)) \subseteq &\ \mathcal{D}(A'),
\end{align*}
for $a,b \in A$. Note that in addition to the morphism property with respect to both products, we require that derivations are mapped to derivations. \\

In \cite{Munthe-KaasWright2008} it was shown that $(\mathcal{OF},\cdot,\graft)$ is the free $D$-algebra and its derivations are exactly the Lie polynomials. These are the elements generated from planar trees $\mathcal{PT} \subset \mathcal{OF}$ by the commutator bracket $[\tau_1,\tau_2]=\tau_1\tau_2-\tau_2\tau_1$. The Lie polynomials are also exactly the elements that are primitive with respect to the shuffle coproduct, meaning, those elements $\omega \in \mathcal{OF}$ satisfying $\Delta_{\shuffle}(\omega)= \emptyset \otimes \omega + \omega \otimes \emptyset$. \\

\begin{remark}
\label{rmk:vectorfields1}
The notion of $D$-algebra has a geometric origin. Indeed, let $M$ be a manifold. It is well-known that one can endow the space $\mathcal{X}M$ of vector fields over $M$ with a post-Lie structure. The extension of this post-Lie structure to a $D$-algebra describes the differential operators. This is an important example and details can be found in, for example, \cite{LundervoldMunthe-Kaas2011,Munthe-KaasLundervold2012,Munthe-KaasWright2008}.
\end{remark}

\subsection{Operads and bialgebras} 
\label{ssec:operads}

We recall the notion of algebraic operad and its link to bialgebras. The reader is referred to \cite{ChapotonLivernet2001,Foissy2017,ManchonHopf} for details. \\

A (symmetric) operad $\mathcal{P}=\oplus_{n=1}^{\infty}\mathcal{P}(n)$, consists of a sequence of vector spaces, $\mathcal{P}(n)$, together with an action of the symmetry group, $\Sigma_n : \mathcal{P}(n) \to \mathcal{P}(n)$, and a map $\circ : \oplus_{n \geq 1} \mathcal{P}^{\otimes n} \otimes \mathcal{P}(n) \to \mathcal{P}$ satisfying the following restriction:
\begin{itemize}
\item $\circ : \mathcal{P}(i_1) \otimes \dots \otimes \mathcal{P}(i_n) \otimes \mathcal{P}(n)  \to \mathcal{P}(i_1 + \dots + i_n)$.
\item There exists an identity element $1 \in \mathcal{P}(1)$ such that $x \circ 1 =x$ and $1 \cdots 1 \circ y = y$ for all $x,y \in \mathcal{P}$.
\item The associativity relation
\begin{align*}
\lefteqn{(x_{1,1}  \cdots  x_{1,n_1} \circ x_1) \cdots (x_{m,1} \cdots x_{m,n_m} \circ x_m) \circ x}\\
&= x_{1,1} \cdots x_{1,n_1} x_{2,1} \cdots x_{2,n_2} x_{3,1} \cdots x_{m,1} \cdots  x_{m,n_m} \circ (x_1 \cdots x_m \circ x)
\end{align*}
is satisfied.
\item The equivariance conditions
\begin{align*}
x_{\sigma^{-1}(1)} \cdots x_{\sigma^{-1}(n)} \circ \sigma(x)
=&\sigma(x_1 \cdots x_n \circ x), & \sigma \in \Sigma_n, \\
\sigma_1(x_1) \cdots  \sigma_n(x_n) \circ x =& (\sigma_1,\ldots,\sigma_n)(x_1  \ldots x_n \circ x), & \sigma_i \in \Sigma_{|x_i|},
\end{align*}
are satisfied. On the right side of the first equality, we interpret $\sigma$ as acting on $\{1,2,\dots,|x_1|+\dots+|x_n| \}$ by permuting the blocks $\{1,\dots,|x_1|\}, \{|x_1|+\dots+|x_j|+1,\dots,|x_1|+\dots+|x_{j+1}|\}, j=1,\dots,n$. In the second equality, we interpret $(\sigma_1,\ldots,\sigma_n)$ as $s_i$ acting on the i:th block.
\end{itemize}

A (right) module\cite{Fresse2007} $\mathcal{M}=\oplus_{n=1}^{\infty}\mathcal{M}(n)$ over an operad $\mathcal{P}$ consists of a sequence of vector spaces, $\mathcal{M}(n)$, together with an action of the symmery group, $\Sigma_n : \mathcal{M}(n) \to \mathcal{M}(n)$, and a map $\circ:  \mathcal{M}(i_1)\otimes\dots \otimes \mathcal{M}(i_n) \otimes \mathcal{P}(n) \to \mathcal{P}(i_1 + \dots + i_n)$ that satisfies associativity and equivariance.\\ 

A bialgebra $(V,\cdot,\Delta,\eta,\epsilon )$ is a vector space $V$ together with an associative multiplication, $\cdot : V \otimes V \to V$ ($x \cdot (y \cdot z)= (x \cdot y) \cdot z $), a coassociative coproduct, $\Delta: V \to V \otimes V$ ($(Id \otimes \Delta)\Delta=(\Delta \otimes Id)\Delta$), a unit map, $\eta: \mathbb{K} \to V$ ($\eta(1)\cdot x=x$), and the counit, $\epsilon: V \to \mathbb{K}$ ($(Id \otimes \epsilon)\Delta=Id=(\epsilon \otimes Id)\Delta$), satisfying the relations:
\begin{align*}
\Delta(x \cdot y)=&\Delta(x) \cdot \Delta(y),\\
\epsilon(x)\epsilon(y)=&\epsilon(x \cdot y),\\
\Delta(\eta(x))=&(\eta \otimes \eta)(x),\\
Id_{\mathbb{K}} =& \epsilon \circ \eta.
\end{align*}
A graded bialgebra is called connected if $\eta$ is an isomorphism between $\mathbb{K}$ and the set of degree zero elements. A Hopf algebra is defined as a bialgebra equipped with an anti-homomorphism $S: V \to V$ called the antipode satisfying
\begin{align*}
\cdot \circ (S \otimes Id)\Delta=\eta \circ \epsilon = \cdot \circ ( Id \otimes S
)\Delta.
\end{align*}
A connected and graded bialgebra is a Hopf algebra. 

Foissy \cite{Foissy2017} describes how to construct bialgebras from operads. The following construction is especially relevant. Let $\mathcal{P}$ be an operad and consider the map $\circ$, the pre-image of each vector space $\mathcal{P}(n)$ under this map is
\begin{align*}
\circ^{-1}(\mathcal{P}(n))
=\bigoplus_{j=1}^{n} \bigoplus_{k_1+\dots+k_j=n} \mathcal{P}(k_1) \otimes \cdots \otimes \mathcal{P}(k_j) \otimes \mathcal{P}(j) .
\end{align*}
Furthermore, we identify the dual space $\mathcal{P}^{\ast}$ with $\mathcal{P}$ by using the canonical dual pairing and consider the map $\Delta: \mathcal{P} \to T(\mathcal{P}) \otimes \mathcal{P}$, where $T(\mathcal{P})$ is the tensor algebra over $\mathcal{P}$, defined by
\begin{align*}
\langle x_1 \cdots x_n \circ x, x' \rangle = \langle x_1 \cdots x_n \otimes x, \Delta(x') \rangle.
\end{align*}
Then:
\begin{align*}
\Delta(\mathcal{P}(n))  \subseteq \bigoplus_{j=1}^{n} \bigoplus_{k_1+\cdots+k_j=n} \mathcal{P}(k_1) \cdots \mathcal{P}(k_j) \otimes \mathcal{P}(j) .
\end{align*}
Foissy showed that $(T(\mathcal{P}),m_{\text{conc}},\Delta)$ is a graded bialgebra, where $m_{\text{conc}}$ is the concatenation product on $T(\mathcal{P})$ and $\Delta$ is multiplicatively extended to $T(\mathcal{P})$ with respect to this product. \\

We conclude by recalling the important notion of bialgebras in cointeraction \cite{Foissy2017,ManchonAbel}:

\begin{definition} 
\label{def::cointeraction}
We say that two bialgebras $(A,\cdot_{\scriptscriptstyle{A}},\Delta_A,\epsilon_A,\eta_A)$, $(B,\cdot_{\scriptscriptstyle{B}},\Delta_B,\epsilon_B,\eta_B)$ are in cointeraction if $B$ is coacting on $A$ via a map $\rho: A \to B \otimes A$ that satisfies:
\begin{align*}
\rho(1_A)=&1_B \otimes 1_A, \\
\rho(x \cdot_{\scriptscriptstyle{A}} y)=&\rho(x) \cdot_{\scriptscriptstyle{B}} \rho(y), \\
(Id \otimes \epsilon_A)\rho=&1_B \epsilon_A, \\
(Id \otimes \Delta_A)\rho=&m_{\scriptscriptstyle{B}}^{1,3}(\rho \otimes \rho)\Delta_A,
\end{align*}
where 
\begin{align*}
m_{\scriptscriptstyle{B}}^{1,3}(a \otimes b \otimes c \otimes d)=a \cdot_{\scriptscriptstyle{B}}c \otimes b \otimes d.
\end{align*}
\end{definition}

\subsection{$B$-series}
\label{ssec:Bseries}

Let $(A,\cdot)$ denote an arbitrary pre-Lie algebra and introduce a fictitious unit $\mathbf{1}$ such that $\mathbf{1} \cdot a = a \cdot \mathbf{1} =a$ for any $a \in A$. As $(\mathcal{T},\curvearrowright)$ is the free pre-Lie algebra, there exists for any element $a \in A$ a unique pre-Lie morphism $F_a: \mathcal{T} \to A$ defined by $F_a (\bullet)=a$. A $B$-series is then defined as a function
\begin{align*}
B(h,a,\alpha)=&\alpha(\emptyset)\mathbf{1}+\sum_{\tau \in \mathcal{T}}h^{v(\tau)}\frac{\alpha(\tau)}{\sigma(\tau)}F_a(\tau),
\end{align*}
where $h \in \mathbb{K}$ is a constant, $v(\tau)$ denotes the function that counts the number of vertices of the input tree $\tau \in \mathcal{T}$, $\alpha: \mathcal{T}\oplus \mathbb{K}\emptyset \to \mathbb{K}$ is a linear function and $\sigma(\tau)$ is the number of symmetries of the tree $\tau$. \\

\begin{remark}\cite{ChartierHairerVilmart2005} \cite{ChartierHairerVilmart2010}
\label{rmk:vectorfields2}
Let $f: \mathbb{R}^n \to \mathbb{R}^n$ be a vector field, then $f$ and its derivatives form a pre-Lie algebra under composition. The typical $B$-series in numerical integration is going to map into this pre-Lie algebra via the pre-Lie algebra morphism given by $F_f(\bullet)=f$. The map $F_f$ is called the elementary differential. \\
If the $B$-series $B(h,f,\alpha)$ is given by a linear map with $\alpha(\emptyset)=1$, then $B(h,f,\alpha)$ is close to the identity map and describes a flow:
\begin{align*}
B(h,f,\alpha)(y)=y+h\alpha(\bullet)f(y)+h^2\alpha(\Forest{[[]]})f'(f(y))+\frac{h^3}{2}\alpha(\Forest{[[][]]})f''(f(y),f(y))+h^3\alpha(\Forest{[[[]]]})f'(f'(f(y)))+\cdots
\end{align*}
These flows can be composed, the result of which can surprisingly be described by a $B$-series. We call this composition of $B$-series. \\
If the $B$-series $B(h,f,\beta)$ is given by a linear map with $\beta(\emptyset)=0$, then $B(h,f,\beta)$ is close to $f$ and describes a vector field. Since this $B$-series is a vector field, it then makes sense to consider something of the form $B(h,B(h,f,\beta),\alpha)$. This is what we call substitution of $B$-series and it turns out that this can be expressed again as a $B$-series in the vector field $f$.
\end{remark}

The results on composition and substitution of $B$-series come from finding appropriate bialgebra structures on the space of forests, $\mathcal{F}$: the Hopf algebra $\mathcal{H}_{CK}=(\mathcal{F},\cdot,\Delta_{CK})$ by Connes and Kreimer \cite{ConnesKreimer1998}, as well as the extraction-contraction bialgebra $\mathcal{H}=(\mathcal{F},\cdot,\Delta_{\mathcal{H}})$ by Calaque, Ebrahimi-Fard and Manchon \cite{CalaqueEbrahimi-FardManchon2011}. These are equal as algebras, both having the commutative concatenation product. The coproduct $\Delta_{CK}$ is defined by admissible edge cuts. Let $\tau \in \mathcal{T}$ be a non-planar rooted tree and let $c$ be a (possibly empty) subset of edges in $\tau$. We say that $c$ is an admissible edge cut if it contains at most one edge from each path in $x$ that starts in the root and ends in a leaf. Removing the edges in $c$ from $\tau$ produces several connected components, the connected component containing the root of $\tau$ will be denoted by $R^c(\tau)$. The concatenation of the remaining connected components will be denoted by $P^c(\tau)$. The coproduct is then given by
\allowdisplaybreaks
\begin{align*}
\Delta_{CK}(\tau)=\sum_{c \text{ admissible cut}} P^c(\tau) \otimes R^c(\tau)+\tau \otimes \emptyset
\end{align*}
on non-planar rooted trees, and extended to forests by
\allowdisplaybreaks
\begin{align*}
\Delta_{CK}(\tau_1\cdots \tau_n)=\Delta_{CK}(\tau_1)\cdots \Delta_{CK}(\tau_n).
\end{align*}
We illustrate this coproduct with a few examples:
\allowdisplaybreaks
\begin{align*}
\Delta_{CK}(\Forest{[[]]})=\ &\emptyset \otimes \Forest{[[]]} + \bullet \otimes \bullet + \Forest{[[]]}\otimes \emptyset, \\
\Delta_{CK}(\Forest{[[][]]})=\ &\emptyset \otimes  \Forest{[[][]]}+2 \bullet \otimes \Forest{[[]]}+\bullet \bullet \otimes \bullet + \Forest{[[][]]} \otimes \emptyset,\\
\Delta_{CK}(\Forest{[[]]}\Forest{[[][]]})=\ & \emptyset \otimes \Forest{[[]]}\Forest{[[][]]}+\bullet \otimes \bullet \Forest{[[][]]}+2 \bullet \otimes \Forest{[[]]}\Forest{[[]]}+3\bullet \bullet \otimes \bullet \Forest{[[]]}+\bullet \bullet \bullet \otimes \bullet \bullet\\
 +& \Forest{[[]]}\otimes \Forest{[[][]]}+2\Forest{[[]]}\bullet \otimes \Forest{[[]]}+\Forest{[[]]}\bullet \bullet \otimes \bullet + \Forest{[[][]]} \otimes \Forest{[[]]} + \Forest{[[][]]} \bullet \otimes \bullet + \Forest{[[]]} \Forest{[[][]]} \otimes \emptyset.
\end{align*}

The coproduct $\Delta_{\mathcal{H}}$ is defined by contractions of subtrees. Let $\tau \in \mathcal{T}$ be a  non-planar rooted tree and let $(\tau_1,\ldots,\tau_n)$ be a spanning subforest of $\tau$, i.e., each $\tau_i$ is a subtree of $\tau$ and each vertex of $\tau$ is contained in exactly one $\tau_i$. We denote by $\tau/\tau_1\cdots \tau_n$ the tree obtained by contracting each subtree to a single vertex. The coproduct, $\Delta_{\mathcal{H}}$, is then given by
\begin{align} \label{eq::DeltaH}
\Delta_{\mathcal{H}}(\tau)=\sum_{(\tau_1,\ldots,\tau_n) \atop \text{ spanning}\ \text{subforest}} \tau_1\cdots \tau_n \otimes \tau/\tau_1 \cdots \tau_n
\end{align}
and extended to forests multiplicatively
\begin{align*}
\Delta_{\mathcal{H}}(\tau_1\cdots \tau_n)=\Delta_{\mathcal{H}}(\tau_1)\cdots \Delta_{\mathcal{H}}(\tau_n).
\end{align*}
We illustrate the coproduct with a few examples:
\allowdisplaybreaks
\begin{align*}
\Delta_{\mathcal{H}}(\Forest{[[[]]]})=\ &\Forest{[[[]]]} \otimes \bullet + 2\Forest{[[]]}\bullet \otimes \Forest{[[]]} + \bullet \bullet \bullet \otimes \Forest{[[[]]]}, \\
\Delta_{\mathcal{H}}(\Forest{[[][]]})=\ & \Forest{[[][]]} \otimes \bullet + 2 \Forest{[[]]} \bullet \otimes \Forest{[[]]}+\bullet \bullet \bullet \otimes \Forest{[[][]]}, \\
\Delta_{\mathcal{H}}(\Forest{[[[]][]]})=\ &\Forest{[[[]][]]} \otimes \bullet + \bullet \Forest{[[[]]]} \otimes \Forest{[[]]} + \bullet \Forest{[[][]]} \otimes \Forest{[[]]}+2\bullet \bullet \Forest{[[]]} \otimes \Forest{[[][]]} +\Forest{[[]]}\Forest{[[]]} \otimes \Forest{[[]]} + \bullet \bullet \Forest{[[]]} \otimes \Forest{[[[]]]}+\bullet \bullet \bullet \bullet \otimes \Forest{[[[]][]]}.
\end{align*}
The two bialgebras $\mathcal{H}_{CK}$ and $\mathcal{H}$ are in cointeraction \cite{CalaqueEbrahimi-FardManchon2011}. We are now ready to recall two important theorems on $B$-series.

\begin{theorem}
Let $\alpha,\beta$ be characters of $\mathcal{H}_{CK}$. Let $m_{CK}$ denote the concatenation product of $\mathcal{H}_{CK}$, then the composition of $B$-series satisfies
\begin{align*}
B(h,a,\beta) \circ B(h,a,\alpha)=B(h,a,\beta \star_{CK} \alpha),
\end{align*}
where $\star_{CK}$ is the convolution product defined in terms of the coproduct of $\mathcal{H}_{CK}$, meaning
\begin{align*}
\beta \star_{CK} \alpha = m_{CK}(\beta \otimes \alpha)\Delta_{CK}.
\end{align*}
\end{theorem}

\begin{theorem} \label{thm::Bsub}
Let $\alpha,\beta : \mathcal{T}\oplus \mathbb{K}\emptyset \to \mathbb{K}$ be linear maps satisfying $\alpha(\emptyset)=0$. Extend $\alpha$ to $\mathcal{H}$ multiplicatively. Then the substitution of $B$-series satisfies
\begin{align*}
B(h,\frac{1}{h}B(h,a,\alpha),\beta	)=B(h,a,\alpha \star_{\mathcal{H}} \beta),
\end{align*}
where $\star_{\mathcal{H}}$ is the convolution product defined in terms of the coproduct of $\mathcal{H}$.
\end{theorem}

\begin{remark}
\label{rmk:vectorfields}
The characters of $\mathcal{H}_{CK}$ form a group under $\star_{CK}$. The linear maps $\alpha$ with $\alpha(\emptyset)=0$ act on this group by $\star_{\mathcal{H}}$, meaning that the maps $\alpha \star_{\mathcal{H}} : \mathcal{F}^{\ast} \to \mathcal{F}^{\ast}$ form a subgroup of the automorphism group over characters of $\mathcal{H}_{CK}$, where $\mathcal{F}^{\ast}$ is the linear dual space of $\mathcal{F}$. The cointeraction between $\mathcal{H}_{CK}$ and $\mathcal{H}$ is vital for this action. This way of seeing a subgroup of the automorphism group over characters of $\mathcal{H}_{CK}$ was used by Bruned, Hairer and Zambotti \cite{BrunedhairerZambotti2019,BrunedhairerZambotti2020} to develop a theory of renormalisation of stochastic partial differential equations.
\end{remark}

\subsection{LB-series}
\label{ssec:LBseries}

Let $\mathbf{D}$ denote an arbitrary $D$-algebra and let $a \in \mathcal{D}(\mathbf{D})$ be a derivation of $\mathbf{D}$. By the freeness property of $(\mathcal{OF},\cdot,\graft)$ as a $D$-algebra, there is a unique $D$-algebra morphism defined by $F_a(\bullet)=a$. A LB-series is then defined as a formal sum
\begin{align*}
LB(a,\alpha)=\sum_{\omega \in \mathcal{OF}} \alpha(\omega)F_a(\omega),
\end{align*}
where $\alpha:\mathcal{OF} \to \mathbb{K}$ is a linear map. We say that $LB(a,\alpha)$, or just $\alpha$, is logarithmic if
\begin{align*}
\alpha(\emptyset)=&\ 0, \\
\alpha(\omega_1 \shuffle \omega_2)=&\ 0,
\end{align*}
for all $\omega_1, \omega_2 \in \mathcal{OF}$. We say that $LB(a,\beta)$, or just $\beta$,	 is exponential if
\begin{align*}
\beta(\emptyset)=&\ 1, \\
\beta(\omega_1 \shuffle \omega_2)=&\ \beta(\omega_1)\beta(\omega_2),
\end{align*}
for all $\omega_1,\omega_2 \in \mathcal{OF}$. \\

Similar to how composition of $B$-series is understood with the help of the Hopf algebra by Connes and Kreimer, we capture composition of LB-series with the help of the Hopf algebra $\mathcal{H}_N$ introduced by Munthe-Kaas and Wright \cite{Munthe-KaasWright2008}. The Hopf algebra $\mathcal{H}_N$ is defined over $\mathcal{OF}$, its multiplication is the shuffle product $\shuffle$, and its coproduct $\Delta_N$ is defined by planar left admissible edge cuts. \\
Let $\tau \in \mathcal{PT}$ be a planar rooted tree and let $c$ be a (possibly empty) subset of edges in $\tau$. We say that $c$ is an admissible planar left cut if it contains at most one edge from each path in $\tau$ from the root to a leaf. Furthermore if $e$ is an edge in $c$, then every edge outgoing from the same vertex as $e$ and that is to the left of $e$ in the planar embedding, is also in $c$. Removing the edges in $c$ from $\tau$ produces several connected components, the one containing the root of $\tau$ will be denoted by $\mathrm{R}^c(\tau)$. Connected components that are cut off from the same vertex will be concatenated to an ordered forest respecting the order, and then the resulting ordered forests will be shuffled, which is denoted by $\mathrm{P}^c(\tau)$. The coproduct $\Delta_N$ is defined by
\begin{align}
\label{MKWcoprod}
\Delta_N(\tau)=\sum_{c \text{ planar left} \atop \text{admissible cut}} \mathrm{P}^c(\tau) \otimes \mathrm{R}^c(\tau) + \tau \otimes 1
\end{align}
on planar rooted trees. It is extended to forests by
\begin{align*}
\Delta_N(\omega)=(Id \otimes B^-)\Delta_N(B^+(\omega)).
\end{align*}

Note that this is dual to the planar Grossman--Larson product \eqref{planarGLprod}, meaning that it also satisfies:
\begin{align}
\label{MKWcoproduct}
\Delta_N(\omega)=\sum_{\omega \text{ is a summand} \atop \text{in } \omega_1 \diamond \omega_2} \omega_1 \otimes \omega_2,
\end{align}
for $\omega,\omega_1,\omega_2$ ordered forests.

\begin{remark}
Note that the sum on the righthand side of \eqref{MKWcoproduct} could have been written running over $ \mathcal{OF}$, using the natural pairing $<\omega_1, \omega_2>=\delta_{\omega_1, \omega_2}$ in the summand. In fact, the duality to the Grossman--Larson product could have been taken as the definition of the coproduct, up to the identification that shuffle products appearing on the lefthand side must be evaluated (giving a linear combination of forests). 
\end{remark}

We illustrate the coproduct \eqref{MKWcoproduct} with a few examples:
\allowdisplaybreaks
\begin{align*}
\Delta_N(\Forest{[[][[]][]]})=&1 \otimes \Forest{[[][[]][]]} + \bullet \otimes \Forest{[[][[]]]}+ \bullet \otimes \Forest{[[][][]]}+ \bullet \shuffle \bullet \otimes \Forest{[[][]]}+ \Forest{[]} \Forest{[[]]} \otimes \Forest{[[]]} + \bullet \Forest{[[]]} \bullet \otimes \bullet + \Forest{[[][[]][]]} \otimes 1, \\
\Delta_N(\Forest{[]}\Forest{[[]]}\Forest{[]})=&1 \otimes \Forest{[]}\Forest{[[]]}\Forest{[]} + \bullet \otimes \Forest{[[]]}\Forest{[]}+\bullet \otimes \bullet \bullet \bullet + \bullet \shuffle \bullet \otimes \bullet \bullet + \Forest{[]}\Forest{[[]]} \otimes \Forest{[[]]}+\Forest{[]}\Forest{[[]]}\Forest{[]}\otimes \Forest{[]} + \Forest{[]}\Forest{[[]]}\Forest{[]} \otimes 1, \\
\Delta_N(\Forest{[[][]]}\Forest{[[]]})=&1 \otimes \Forest{[[][]]}\Forest{[[]]} +\bullet \otimes \Forest{[[]]}\Forest{[[]]}+\bullet \otimes \Forest{[[][]]}\Forest{[]} + \bullet \shuffle \bullet \otimes \Forest{[[]]}\Forest{[]} + \bullet \bullet \otimes \Forest{[]}\Forest{[[]]}+ \bullet \bullet \shuffle \bullet \otimes \bullet \bullet\\
 +& \Forest{[[][]]}\otimes \Forest{[[]]}+\Forest{[[][]]}\shuffle \bullet \otimes \bullet + \Forest{[[][]]}\Forest{[[]]} \otimes 1.
\end{align*}

Before we move on to state the composition theorem for LB-series, we want to remark on $D$-algebras generated by vector fields over a manifold. 

\begin{remark}
Typically in applications of {\rm{LB}}-series in geometric integration over a manifold, the $D$-algebra morphism $F_a$ will map from ordered forests $(\mathcal{OF},\cdot,\graft)$ into a $D$-algebra generated by vector fields over a manifold. In this case, logarithmic {\rm{LB}}-series describe vector fields, which are the derivations in the target $D$-algebra. Exponential {\rm{LB}}-series describe flows on the manifold. In this case, composition is understood as composition of flows and this is what we mean concretely by composition of {\rm{LB}}-series.
\end{remark}

\begin{theorem}\cite{Munthe-KaasWright2008} \label{thm::1}
Let $\alpha,\beta$ be characters of $\mathcal{H}_N$, then the composition of {\rm{LB}}-series satisfies
\begin{align*}
LB(a,\beta)\circ LB(a,\alpha)=LB(a,\beta \star_N \alpha),
\end{align*}
where $\star_N$ is the convolution product defined in terms of the coproduct \eqref{MKWcoproduct} of $\mathcal{H}_N$. This describes a group structure on the set of exponential {\rm{LB}}-series.
\end{theorem}

By substituting a LB-series into another LB-series, we mean something of the form $LB(LB(a,\alpha),\beta)$, i.e., replacing the derivation $a$ in the target $D$-algebra by another LB-series expressed in $a$. This only makes sense if the LB-series $LB(a,\alpha)$ is again a derivation, which happens exactly when it is logarithmic. The aim of the sequel is now to find a Lie--Butcher version of Theorem \ref{thm::Bsub}, which describes substitution in $B$-series.

\section{A substitution operad of non-planar rooted trees} 
\label{section::2}

In this section we shall construct an operad of non-planar rooted trees that is dual to the coproduct $\Delta_{\mathcal{H}}$. We then see in Proposition \ref{prop::6} that this construction amounts to a different way of describing the pre-Lie operad defined by Chapoton and Livernet in \cite{ChapotonLivernet2001}. \\

Let $\mathcal{T}_n$ denote the vector space spanned by all non-planar rooted trees with exactly $n$ vertices, together with a bijection between the set $\{1,\dots,n \}$ and the vertices of a tree in $\mathcal{T}_n$. We consider this bijection as a labelling of the vertices. Let the symmetric group $\Sigma_n$ act on $\mathcal{T}_n$ by permuting the labels of the vertices. Write $[x]$ for the orbit of $x \in \mathcal{T}_n$ under $\Sigma_n$. We consider this as an unlabelled tree. Define the equivalence relation $x \sim y \iff [x]=[y]$ on $\mathcal{T}_n$. Let $\hat{\mathcal{T}}=\sum_{n=1}^{\infty}\mathcal{T}_n$, then $\hat{\mathcal{T}}/\sim$ can be identified with $\mathcal{T}$. This identification is the bosonic Fock functor \cite{aguiar2010monoidal}. \\
We will now define an operad over $\hat{\mathcal{T}}$, consider $x \in \mathcal{T}_n$ and $x_1,\dots,x_n \in \hat{\mathcal{T}}$. Each of the trees in $x,x_1,\dots,x_n$ has a factorization in terms of the single-vertex tree and the grafting product in the free pre-Lie algebra. Define
\begin{align*}
x_1 \cdots x_n \circ x
\end{align*}
to be the result of replacing each occurrence of the single-vertex tree corresponding to vertex number $i$ in the factorization of $x$ by (the factorization of) $x_i$, for all $i=1,\dots,n$. This is well-illustrated by an example.

\begin{example} \label{example::1}
Let
\allowdisplaybreaks
\begin{align*}
x=&\Forest{[1[3][2]]}=\ \Forest{[2]} \curvearrowright ( \Forest{[3]} \curvearrowright \Forest{[1]}) - (\Forest{[2]} \curvearrowright \Forest{[3]}) \curvearrowright \Forest{[1]}, \\
x_1=&\Forest{[1[2]]}, \; x_2= \Forest{[3[2][1[4]]]}, \; x_3=\Forest{[2[1[4[3]]]]}.
\end{align*}
Then:
\allowdisplaybreaks
\begin{align*}
x_1x_2x_3 \circ x =&\ x_2 \curvearrowright ( x_3 \curvearrowright x_1) - (x_2 \curvearrowright x_3) \curvearrowright x_1 \\
=&\ \Forest{[1[2][8[7[10[9]]]][5[3][4[6]]]]}+\Forest{[1[2[5[3][4[6]]]][8[7[10[9]]]]]}+\Forest{[1[2[8[7[10[9]]]]][5[3][4[6]]]]}+\Forest{[1[2[8[7[10[9]]]][5[3][4[6]]]]]}.
\end{align*}
\end{example}

The following proposition states that the combinatorial description of the operad is that of the pre-Lie operad.

\begin{proposition} \label{prop::6}
The expression
\begin{align*}
x_1 \cdots x_n \circ x
\end{align*}
evaluates to the sum of all possible trees obtained by replacing vertex $i$ in $x$ by the tree $x_i$, for all $i=1,\dots,n$. The incoming edge to vertex $i$ becomes incoming to the root of $x_i$. The edges outgoing from vertex $i$ become outgoing from any vertex of $x_i$.
\end{proposition}

\begin{proof}
Recall the definition of the free pre-Lie product, $x_i \curvearrowright x_j$ is the sum of all trees obtained by adding an edge from any vertex in $x_j$ to the root of $x_i$. Furthermore note that there is an edge from vertex $j$ to vertex $i$ in $x$, if and only if vertex $i$ is pre-Lie grafted onto a subtree of $x$ that contains the vertex $j$. Hence when vertex $i$ is replaced by $x_i$ and vertex $j$ is replaced by $x_j$, we get $x_i$ pre-Lie grafted onto $x_j$.
\end{proof}

We continue with Example \ref{example::1} to illustrate the combinatorial picture.

\begin{example} \label{example::2}
The expression
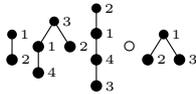
\begin{align*}
\Forest{[1[2]]} \Forest{[3[2][1[4]]]}\Forest{[2[1[4[3]]]]} \circ \Forest{[1[3][2]]}
\end{align*}
can be visualized by putting each input tree inside its respective vertex:\\
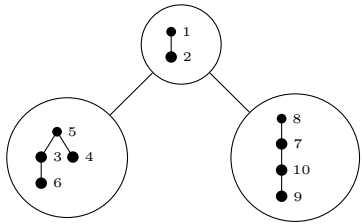
\begin{center}
\begin{tikzpicture}
\node[circle,draw] at (0,0) (node1) {\Forest{[1[2]]}};
\node[circle,draw] at(-1.5,-1.5) (node2) {\Forest{[5[4][3[6]]]}};
\node[circle,draw] at (1.5,-1.5) (node3) {\Forest{[8[7[10[9]]]]}};

\draw[-] (node1) -- (node2);
\draw[-] (node1) -- (node3);
\end{tikzpicture}
\end{center}

The four terms we see in Example \ref{example::1} come from the four different ways of grafting the edges from $\Forest{[1[3][2]]}$ onto some vertex of $\Forest{[1[2]]}$.
\end{example}

The visualization in Example \ref{example::2} below illustrates a duality to the coproduct $\Delta_{\mathcal{H}}$, which we will make precise in the following proposition:

\begin{proposition} \label{prop::1}
Let $x \in \hat{\mathcal{T}}$ be a rooted tree, then
\begin{align*}
\Delta_{\mathcal{H}}([x])=&\sum_{[x] \text{ is a summand } \atop \text{in } [x_1\cdots x_n \circ x']} \frac{1}{n!} [x_1]\cdots [x_n] \otimes [x'],
\end{align*}
for $x_1,\ldots,x_n,x'$ labeled trees.
\end{proposition}

\begin{proof}
If $[x]$ is a summand in $[x_1\cdots x_n \circ x']$, then $[x_1] \cdots [x_n]$ is a spanning subforest of $[x]$. Furthermore, $[x]/[x_1] \cdots [x_n] = [x']$. \\
Conversely, let $[x_1] \cdots [x_n]$ be an arbitrary spanning subforest of $[x]$. Then for any labelling of the vertex of $[x]/[x_1] \cdots [x_n]$, there is an ordering of $x_1 \cdots x_n$ such that the rooted tree that was contracted into vertex $i$ is in position $i$, for $i=1,\ldots,n$. Then, for any choice of labelling on each $x_i$, we have that $[x_1 \cdots x_n \circ [x]/[x_1] \cdots [x_n]]$ is the sum of all possible trees obtained by replacing each vertex in $[x]/[x_1] \cdots [x_n]$ by the rooted tree it was contracted from. In particular, $[x]$ must be a summand in this sum. \\
Finally, we note that each of the $n!$ different labellings of $[x']$ gives us the same spanning subforest.
\end{proof}

\begin{remark}
The coproduct construction in Proposition \ref{prop::1} is a symmetrization of the bialgebra construction by Foissy that we discussed in Section \ref{ssec:operads}. The observation that the bialgebra $\mathcal{H}$ can be obtained this way was mentioned without proof by Foissy in \cite{Foissy2017}.
\end{remark}

\section{A substitution operad of planar rooted trees}
\label{sec:planarSubsOperad}

Motivated by the observation that the coproduct $\Delta_{\mathcal{H}}$ describing substitution in $B$-series can be seen as dual to a substitution operad of non-planar trees, we shall in this section construct a substitution operad of Lie polynomials of planar rooted trees. Recall that substitution in LB-series is described by a map that sends the single-vertex tree to a Lie polynomial and then generates to be a $D$-algebra morphism. Furthermore, recall that the Lie polynomials are in bijection with the underlying free post-Lie algebra, which can be represented by formal Lie brackets on planar rooted trees. \\

Let $Lie(\mathcal{PT})_n$ denote the vector space spanned by all Lie brackets over planar trees, such that there are in total exactly $n$ vertices in the brackets, together with a bijection between the set $\{1,\ldots,n \}$ and the vertices. This bijection provides a labelling of the vertices. Let the symmetric group $\Sigma_n$ act on $Lie(\mathcal{PT})_n$ by permuting the labels of the vertices. Write $[\omega]$ for the orbit of $\omega$ under $\Sigma_n$, we consider this as an unlabelled element. Define the equivalence relation $\omega_1 \sim \omega_2 \iff [\omega_1]=[\omega_2]$ on $Lie(\mathcal{PT})_n$ and let $\widehat{Lie(\mathcal{PT})}=\sum_{n=1}^{\infty}Lie(\mathcal{PT})_n$, then $\widehat{Lie(\mathcal{PT})}/\sim$ can be identified with $Lie(\mathcal{PT})$. \\

We will now define an operad over $\widehat{Lie(\mathcal{PT})}$, consider $\omega \in Lie(\mathcal{PT})_n$ and $\omega_1,\ldots,\omega_n \in \widehat{Lie(\mathcal{PT})}$. Each element of $\omega,\omega_1,\ldots,\omega_n$ can be expressed as a monomial in terms of the Lie bracket, the post-Lie grafting and the single-vertex tree. Define
\begin{align*}
\omega_1 \cdots \omega_n \circ \omega
\end{align*}
to be the result of replacing each occurrence of the single-vertex tree corresponding to vertex number $i$ in the expression for $\omega$ by the expression for $\omega_i$, for all $i=1,\ldots,n$. This is well-illustrated by an example.

\begin{example}
Let
\allowdisplaybreaks
\begin{align*}
\omega=[\Forest{[1]},[\Forest{[2[3]]},\Forest{[4[6][5]]}]]=[\Forest{[1]},[\Forest{[3]}\graft \Forest{[2]},\Forest{[5]}\graft (\Forest{[6]} \graft \Forest{[4]})-(\Forest{[5]} \graft \Forest{[6]})\graft \Forest{[4]}]],
\end{align*}
then:
\allowdisplaybreaks
\begin{align*}
\omega_1\cdots \omega_6 \circ \omega=[\omega_1,[\omega_3\graft \omega_2,\omega_5\graft (\omega_6 \graft \omega_4)-(\omega_5 \graft \omega_6)\graft \omega_4]].
\end{align*}
\end{example}

\begin{proposition} \label{prop::7}
$(\widehat{Lie(\mathcal{PT})},\circ)$ is an operad.
\end{proposition}

\begin{proof}
It is clear that the single-vertex tree is an identity element. It remains to show associativity and equivariance. \\
We have that
\begin{align*}
(\omega_{1,1} \cdots \omega_{1,n_1} \circ \omega_1) \cdots (\omega_{n,1}\cdots \omega_{n,n_l} \circ \omega_n) \circ \omega
\end{align*}
is the expression obtained by replacing vertex $i$ in $\omega$ by $\omega_{i,1} \cdots \omega_{i,n_i} \circ \omega_i$, for $i=1,\dots,l$. Furthermore we have that
\begin{align*}
\omega_{1,1} \cdots \omega_{1,n_1} \omega_{2,1} \cdots \omega_{2,n_2} \omega_{3,1} \dots \omega_{n,n_l} \circ (\omega_1 \cdots \omega_n \circ \omega   )
\end{align*}
is the expression obtained by first replacing vertex $i$ in $\omega$ by $\omega_i$ and then replacing $\omega_i$ by $\omega_{i,1} \dots \omega_{i,n_i} \circ \omega_i$, for $i=1,\ldots,l$. Hence we have associativity.\\
Equivariance follows from the observation that the resulting unlabelled elements only depend on a coupling between vertices and input elements. 
\end{proof}

We get the following combinatorial description of the operad.

\begin{proposition} \label{prop::3}
The expression
\begin{align*}
\omega_1 \cdots \omega_n \circ \omega
\end{align*}
evaluates to the sum of all Lie brackets of planar trees obtained by replacing vertex $i$ in $\omega$ by $\omega_i$. Edges outgoing from vertex $i$ becomes outgoing from any vertex of $\omega_i$, such that their left-to-right order is preserved when they are outgoing from the same vertex. Edges incoming to vertex $j$ becomes incoming to the root of $\omega_j$ if $\omega_j$ is a tree. If $\omega_j$ is a Lie bracket, we interpret the incoming edge as one edge per root and graft all trees in $\omega_j$ onto the same vertex in all possible left-to-right orders prescribed by the interpretation $[\omega_a,\omega_b]=\omega_a\omega_b-\omega_b\omega_a$.
\end{proposition}
\begin{proof}
This is completely analogous to the statement and proof of Proposition \ref{prop::6}. Note that there is now also a factor of planarity and that all grafting is done leftmost, note also that the identity
\begin{align*}
[\omega_1,\omega_2] \graft \omega_3 = \omega_1 \graft (\omega_2 \graft \omega_3) - (\omega_1 \graft \omega_2) \graft \omega_3 - \omega_2 \graft (\omega_1 \graft \omega_3) + (\omega_2 \graft \omega_1) \graft \omega_3
\end{align*}
is used to deal with replacing a non-root vertex by a Lie bracket.
\end{proof}

\begin{corollary}
The operad $(\hat{Lie(PT)},\circ)$ is the post-Lie operad.
\end{corollary}
\begin{proof}
In proposition \ref{prop::3}, we recovered the combinatorial description of the post-Lie operad from \cite{Silva2018}.
\end{proof}

Motivated by Proposition \ref{prop::1}, we are now going to define a coproduct as the dual of this substitution operad.

\begin{definition}
Let $S(Lie(\mathcal{PT}))$ denote the symmetric algebra over the vector space $Lie(\mathcal{PT})$, meaning that it is the vector space of all unordered sequences of Lie polynomials of ordered forests together with a commutative concatenation product. We denote the commutative concatenation of $\omega_1$ and $\omega_2$ by $\omega_1\centerdot\omega_2$. We furthermore denote the identity element of this product by $1$. \\
Let $\omega \in \widehat{Lie(\mathcal{PT})}$ be a Lie polynomial. We define the coproduct $\Delta_{\mathcal{Q}}: S(Lie(\mathcal{PT})) \to S(Lie(\mathcal{PT})) \otimes S(Lie(\mathcal{PT}))$ by
\begin{align*}
\Delta_{\mathcal{Q}}([\omega])=\sum_{[\omega] \text{ is a summand in } \atop [\omega_1 \cdots \omega_n \circ \omega']}\frac{1}{n!} [\omega_1]\centerdot \cdots \centerdot[\omega_n] \otimes [\omega'],
\end{align*}
for $\omega_1,\dots,\omega_n,\omega'$ labelled Lie monomials. The coproduct is then extended to $S(Lie(\mathcal{PT}))$ multiplicatively.
\end{definition}

We conclude the section with technical details.

\begin{proposition}
Let the linear map $\epsilon: S(Lie(\mathcal{PT})) \to \mathbb{K}$ be defined by
\begin{align*}
\epsilon(\omega)=\begin{cases*}
1  \text{ if } \omega=1 \text{ or } \omega=\bullet, \\
0  \text{ otherwise}
\end{cases*}
\end{align*}
for $\omega$ a Lie monomial, and extended to $S(Lie(\mathcal{PT}))$ by
\begin{align*}
\epsilon(\omega_1\centerdot \omega_2)=\epsilon(\omega_1)\epsilon(\omega_2).
\end{align*}
Then $(S(Lie(\mathcal{PT})),\Delta_{\mathcal{Q}},\epsilon)$ is a coalgebra.
\end{proposition}

\begin{proof}
The coassociativity of $\Delta_{\mathcal{Q}}$ follows immediately from the associativity of $(\widehat{Lie(\mathcal{PT})},\circ)$. It remains to show that $\epsilon$ is a counit, meaning that it satisfies
\begin{align*}
(Id \otimes \epsilon)\Delta_{\mathcal{Q}}=Id=(\epsilon \otimes Id)\Delta_{\mathcal{Q}}.
\end{align*}
We have that $(Id \otimes \epsilon)\Delta_{\mathcal{Q}}(\omega_1 \centerdot \cdots \centerdot \omega_n)$ is nonzero only on the unique term where every $\omega_i$, $i=1,\dots,n$, is contracted into a single vertex. Hence this has to be the identity. Similarly, we have that $(\epsilon \otimes Id)\Delta_{\mathcal{Q}}(\omega)$ is non-zero only on the unique term where every contracted subforest of $\omega$ consists of a single vertex. Hence this expression evaluates to $\omega$.
\end{proof}

\begin{proposition} 
\label{prop::2}
Define the function $| \cdot | : Lie(\mathcal{PT}) \to \mathbb{N}$ by $| \omega | = \#\{\text{vertices in }\omega \}-1$. Then
\begin{align*}
|[\omega_1 \cdots \omega_n \circ \omega   ]|=|[\omega_1]| + \cdots +|[\omega_n]|+|[\omega]|.
\end{align*}
\end{proposition}

\begin{proof}
Since every vertex in $[\omega_1 \cdots \omega_n \circ \omega]$ comes from a vertex in an $\omega_i$, we have that
\begin{align*}
|[\omega_1 \cdots \omega_n \circ \omega ]|+1=&(|[\omega_1]|+1) + \cdots +(|[\omega_n]|+1) \\
=&(|[\omega_1]| + \cdots +|[\omega_n]|+|[\omega]|)+n.
\end{align*}
Hence
\begin{align*}
|[\omega_1 \cdots \omega_n \circ \omega ]|=&|[\omega_1]| + \cdots +|[\omega_n]|+(n-1) \\
=&|[\omega_1]| + \cdots +|[\omega_n]|+|[\omega]|.
\end{align*}
\end{proof}

\begin{corollary}
The function $| \cdot |$ defined in Proposition \ref{prop::2} induces a grading on the bialgebra $\mathcal{Q}=(S(Lie(\mathcal{PT})),\centerdot,\Delta_{\mathcal{Q}},1,\epsilon)$.
\end{corollary}

\begin{proof}
Let $\omega \in Lie(\mathcal{PT})$ and consider the term $\omega_1 \centerdot \dots \centerdot \omega_n \otimes \omega'$ in $\Delta_{\mathcal{Q}}(\omega)$. By Proposition \ref{prop::2}, we have that $|\omega_1| + \dots +|\omega_n|+|\omega'|=|\omega|$. Extend $|\cdot |$ to $S(Lie(\mathcal{PT}))$ by $|\omega_1 \centerdot \dots \centerdot \omega_n|=|\omega_1| + \dots +|\omega_n|$, then $|\omega|=|\omega_{(1)}|+|\omega_{(2)}|$ and $S(Lie(\mathcal{PT}))=\oplus_{n=1}^{\infty} \{\omega \in S(Lie(\mathcal{PT})): |\omega|=n  \}$ defines a grading.
\end{proof}

\section{Two cointeracting bialgebras}
\label{section::3}

In this section we are going to describe a cointeraction between the Munthe-Kaas--Wright Hopf algebra $\mathcal{H}_N=(\mathcal{OF},\shuffle,\Delta_N,\emptyset,\delta(\emptyset))$, and the bialgebra $\mathcal{Q}$ defined in the previous section. This means that we will define a coaction $\rho: \mathcal{H}_N \to \mathcal{Q} \otimes \mathcal{H}_N$ that satisfies the conditions of Definition \ref{def::cointeraction}. Informally, we want this coaction to be the opposite of substituting Lie polynomials into the vertices of forests.

\begin{remark}
	There exists an isomorphism \cite{FloystadMunthe-Kaas} $(S(Lie(\mathcal{PT})),\centerdot) \cong (\mathcal{OF},\shuffle)$. This means that our coproduct $\Delta_{\mathcal{Q}} : \mathcal{Q} \to \mathcal{Q} \otimes \mathcal{Q}$ can already be seen as a coaction $\mathcal{H}_N \to \mathcal{Q} \otimes \mathcal{H}_N$ via identification by this isomorphism. The isomorphism is however non-trivial. We will prefer to describe the coaction by using a module over an operad.\\ 
	Note that, in the non-planar case, we had an isomorphism between $\mathcal{H}_{CK}$ and $\mathcal{H}$ as algebras. This is now generalized also to the planar case.
\end{remark}

Recall that $\mathcal{OF}$ is the universal enveloping algebra of $Lie(\mathcal{PT})$, hence there is an inclusion $Lie(\mathcal{PT}) \subset \mathcal{OF}$ given by the commutator $[\tau_1,\tau_2]=\tau_1\tau_2 - \tau_2\tau_1$. We will use this inclusion to define what it means to substitute elements of $Lie(\mathcal{PT})$ into vertices of $\mathcal{OF}$. \\

Let $\mathcal{OF}_n$ denote the vector space spanned by all ordered forests with exactly $n$ vertices, together with a bijection between the set $\{1,\ldots,n \}$ and the vertices of the forest. This bijection provides a labelling of the vertices. Let the symmetric group $\Sigma_n$ act on $\mathcal{OF}_n$ by permuting the labels of the vertices. Write $[\omega]$ for the orbit of $\omega$ under $\Sigma_n$, we consider this as an unlabelled forest. Define the equivalence relation $\omega_1 \sim \omega_2 \iff [\omega_1]=[\omega_2]$ on $\mathcal{OF}_n$. Let $\widehat{\mathcal{OF}}=\sum_{n=1}^{\infty}\mathcal{OF}_n$, then $\widehat{\mathcal{OF}}/\sim$ can be identified with $\mathcal{OF}$. This identification is the bosonic Fock functor \cite{aguiar2010monoidal}. \\

We now define composition maps $\circ :\oplus_{n\geq 1} \widehat{Lie(\mathcal{PT})}^{\otimes n} \otimes \mathcal{OF}_n \to \widehat{\mathcal{OF}}$. Let $\omega_1 \cdots \omega_n \in \widehat{Lie(\mathcal{PT})}$ and $\omega' \in \mathcal{OF}_n$, then $\omega'$ can be expressed in terms of the single-vertex tree, associative concatenation and non-associative $D$-algebra product $\graft$. Furthermore, all of $\omega_1 \cdots \omega_n$ can be expressed in the same way via the inclusion $\widehat{Lie(\mathcal{PT})} \subset \widehat{\mathcal{OF}}$. By the composition
\begin{align*}
\omega_1 \cdots \omega_n \circ \omega'
\end{align*}
we will mean the expression obtained by replacing each occurence of vertex number $i$ in the expression of $\omega'$ by the expression for $\omega_i$, $i=1,\dots,n$.

\begin{proposition}
The above composition turns $\widehat{\mathcal{OF}}$ into a (right) module over $\widehat{Lie(\mathcal{PT})}$.
\end{proposition}
\begin{proof}
We first note that the composition is well-defined, since expressing a forest $\omega'$ in terms of associative concatenation, the non-associative product $\graft$ and the single-vertex tree is unique up to rewritings by the identities
\begin{align*}
\tau \graft (\omega_1\omega_2)=&(\tau\graft \omega_1)\omega_2 + \omega_1(\tau \graft \omega_2), \\
\tau \graft (\omega_1 \graft \omega_2)=&(\tau \omega_1)\graft \omega_2 + (\tau \graft \omega_1) \graft \omega_2,
\end{align*}
for $\tau$ a derivation and $\omega_1,\omega_2$ arbitrary. If $\tau$ is a derivation, it is a Lie polynomial. Since Lie polynomials get mapped to Lie polynomials, the above identities can be applied post-substitution. Hence $\circ$ commutes with rewriting and must be well-defined. \\
Proving associativity and equivariance can be done in the same way as in Proposition \ref{prop::7}.
\end{proof}

\begin{remark}
Note that if we attempt to turn $\widehat{\mathcal{OF}}$ into an operad by replacing the single-vertex tree by an arbitrary forest, we are no longer mapping Lie polynomials to Lie polynomials. This means that computing $\omega_1 \cdots \omega_n \circ \omega'$ can give different results depending on how you choose to express $\omega'$.
\end{remark}

We get the following combinatorial picture.

\begin{proposition}
The expression
\begin{align*}
\omega_1 \cdots \omega_n \circ \omega
\end{align*}
evaluates to the sum of all possible forests obtained by replacing vertex $i$ in $\omega$ by the forest $\omega_i$, for all $i=1,\dots,n$. The incoming edge to vertex $i$ becomes one incoming edge to each of the roots of $\omega_i$. The edges outgoing from vertex $i$ become outgoing from any vertex of $\omega_i$. The left to right ordering of edges outgoing from the same vertex $i$ is preserved whenever the edges end up on the same vertex in $\omega_i$.
\end{proposition}

\begin{proof}
Recall the definition of the non-associative product in the free $D$-algebra, $\omega_a \graft \omega_b$ is the sum of all ways to add one edge incoming to each root of $\omega_a$ that are outgoing in the leftmost position from any vertex of $\omega_b$ and such that if two trees from $\omega_a$ are grafted on the same vertex of $\omega_b$, then their pairwise order in the planar embedding is preserved. Furthermore note that there is an edge from vertex $j$ to vertex $i$ in $\omega$, if and only if vertex $i$ is $\graft$-grafted onto a subtree of $\omega$ that contains the vertex $j$. Hence when vertex $i$ is replaced by $\omega_i$ and vertex $j$ is replaced by $\omega_j$, we get $\omega_i$ grafted onto $\omega_j$ by the product $\graft$.
\end{proof}

We are now ready to define the coaction. Let the coaction $\rho: \mathcal{H}_N \to \mathcal{Q} \otimes \mathcal{H}_N$ be defined by
\begin{align*}
\rho(\emptyset)=&1 \otimes \emptyset, \\
\rho([\omega])=&\sum_{[\omega] \text{ is a summand} \atop \text{in } [\omega_1 \cdots \omega_n \circ \omega']}\frac{1}{n!} [\omega_1]\centerdot \cdots \centerdot[\omega_n] \otimes [\omega'],
\end{align*}
for $\omega_1,\dots,\omega_n$ labeled Lie monomials and $\omega'$ a labeled forest. We now want to give a few example computations of this coaction:

\begin{example}
\begin{align*}
\rho(\Forest{[]}\Forest{[[]]})=&[\Forest{[]},\Forest{[[]]}] \otimes \Forest{[]} + \Forest{[]} \centerdot \Forest{[[]]} \otimes \Forest{[]}\Forest{[]} + \Forest{[]} \centerdot \Forest{[]} \centerdot \Forest{[]} \otimes \Forest{[]}\Forest{[[]]}, \\
\rho(\Forest{[]}\Forest{[[]]}\Forest{[]})=&[[\Forest{[]},\Forest{[[]]}],\Forest{[]}] \otimes \bullet +[\Forest{[]},[\Forest{[[]]},\Forest{[]}]] \otimes \bullet +
 \bullet \centerdot \Forest{[[]]}\centerdot \bullet \otimes \bullet \bullet \bullet\\
  +& [\Forest{[]},\Forest{[[]]}]\centerdot \bullet \otimes \bullet \bullet+\bullet \centerdot [\Forest{[[]]},\bullet] \otimes \bullet \bullet + \bullet \centerdot \bullet \centerdot \bullet \centerdot \bullet \otimes \Forest{[]}\Forest{[[]]}\Forest{[]}, \\
\rho(\Forest{[[][[][]]]})=&\Forest{[[][[][]]]} \otimes \bullet + \Forest{[[]]} \centerdot \bullet \centerdot \bullet \centerdot \bullet \otimes 3\Forest{[[][][]]} + \Forest{[[]]}\centerdot \bullet \centerdot \bullet \centerdot \bullet \otimes \Forest{[[][[]]]}\\
 +& \Forest{[[[]]]} \centerdot \bullet \centerdot \bullet \otimes 2 \Forest{[[][]]}+\Forest{[[][]]} \centerdot \bullet \centerdot \bullet \otimes \Forest{[[][]]}+\Forest{[[[][]]]} \centerdot \bullet \otimes \Forest{[[]]}\\
 +&\Forest{[[][[]]]} \centerdot \bullet \otimes \Forest{[[]]} + [\Forest{[[][]]},\bullet] \centerdot \bullet \otimes \Forest{[[]]} + [\Forest{[[]]},\bullet] \centerdot \bullet \centerdot \bullet \otimes \Forest{[[[]]]}.
\end{align*}
\end{example}

\begin{notation}
We shall in the sequel write $\omega_1 \cdots \omega_n \circ \omega'$ also when the elements are unlabeled. In this case, we mean it as a sum over all ways to pair vertices in $\omega'$ with input elements.
\end{notation}

\begin{proposition} 
\label{prop::cointercation}
The coaction $\rho$ defines a cointeraction between $\mathcal{Q}$ and $\mathcal{H}_N$.
\end{proposition}

\begin{proof}
We first prove the identity:
\begin{align*}
(Id \otimes \Delta_N)\rho(\omega)=m_\centerdot^{1,3}(\rho \otimes \rho)\Delta_N(\omega).
\end{align*}
The identity is clear when $\omega = \emptyset$. Suppose that $\omega$ is a non-empty forest, then we have
\begin{align*}
(Id \otimes \Delta_N)\rho(\omega)=\sum_{\substack{\omega \text{ is a summand} \atop \text{in } \omega_1\dots \omega_n \circ (\omega' \diamond \omega'') }} \omega_1 \centerdot \dots \centerdot \omega_n \otimes \omega' \otimes \omega''.
\end{align*}
Furthermore
\begin{align*}
\lefteqn{m^{1,3}(\rho \otimes \rho)\Delta_N(\omega)}\\
&= 
m_\centerdot^{1,3}(\rho \otimes \rho)
\sum_{\omega \text{ is a summand} \atop \text{in } \omega' \diamond \omega''} \omega' \otimes \omega'' \\
&=
\sum_{\omega \text{ is a summand} \atop \text{in } \omega' \diamond \omega''}
\sum_{{\omega' \text{ is a summand} \atop \text{in } \omega_1' \cdots \omega_{k_1}' \circ \bar{\omega}'}}\sum_{{\omega'' \text{ is a summand} \atop \text{in } \omega_1'' \cdots \omega_{k_2}'' \circ \bar{\omega}''}} \omega_1' \centerdot \cdots \centerdot \omega_{k_1}' \centerdot \omega_1'' \centerdot \cdots \centerdot \omega_{k_2}'' \otimes \bar{\omega}' \otimes \bar{\omega}'' \\
&=\sum_{{\omega \text{ is a summand} \atop \text{in } (\omega_1' \cdots \omega_{k_1}' \circ \bar{\omega}') \diamond (\omega_1'' \cdots \omega_{k_2}'' \circ \bar{\omega}'')}} \omega_1' \centerdot \cdots \centerdot \omega_{k_1}' \centerdot \omega_1'' \centerdot \cdots \centerdot \omega_{k_2}'' \otimes \bar{\omega}' \otimes \bar{\omega}' \\
&=\sum_{{\omega \text{ is a summand} \atop \text{in } \omega_1\cdots \omega_n \circ (\omega' \diamond \omega'') }} \omega_1 \centerdot \cdots \centerdot \omega_n \otimes \omega' \otimes \omega'',
\end{align*}
which proves the identity. \\
Next we need to prove compatibility with the shuffle product, i.e., the identity:
\begin{align*}
\rho(\omega_a \shuffle \omega_b)=\rho(\omega_a) \rho(\omega_b).
\end{align*}
Recall that, using the natural identification $<\omega_1,\omega_2>=\delta_{\omega_1,\omega_2}$, we can write 
$$
	\omega_a \shuffle \omega_b
	= \sum_{\omega_a \otimes \omega_b \text{ is a} \atop \text{summand in } \Delta_{\shuffle}(\omega)} \omega
	= \sum_{\omega} <\Delta_{\shuffle}(\omega), \omega_a \otimes \omega_b> \omega.
$$
We then have:
\begin{align*}
\lefteqn{\rho(\omega_a \shuffle \omega_b)
= \rho\Big(\sum_{\omega_a \otimes \omega_b \text{ is a} \atop \text{summand in } \Delta_{\shuffle}(\omega)} \omega \Big)} \\
&= 
\sum_{\omega_a \otimes \omega_b \text{ is a} \atop \text{summand in } \Delta_{\shuffle}(\omega)} 
\sum_{\omega \text{ is a summand} \atop \text{in } \omega_1 \cdots \omega_m \circ \omega'} \omega_1 \centerdot \cdots \centerdot \omega_m \otimes \omega' \\
&=
\sum_{\omega_a \otimes \omega_b \text{ is a summand} \atop \text{in } \Delta_{\shuffle}(\omega_1 \cdots \omega_m \circ \omega')} \omega_1 \centerdot \cdots \centerdot \omega_m \otimes \omega' \\
&=
\sum_{\omega_a \otimes \omega_b \text{ is a summand}  \atop \text{in } \Delta_{\shuffle}(\omega_1 \cdots \omega_m \circ \tau_1\cdots\tau_k)} \omega_1 \centerdot \dots \centerdot \omega_m \otimes \tau_1\cdots \tau_k \\
&=
\sum_{\omega_a \otimes \omega_b \text{ is a summand in } \atop \Delta_{\shuffle}((\omega_{1,1}\cdots \omega_{1,n_1} \circ \tau_1) \cdots (\omega_{k,1} \cdots \omega_{k,n_k} \circ \tau_k))} \omega_{1,1} \centerdot \cdots \centerdot \omega_{k,n_k} \otimes \tau_1\cdots\tau_k \\
&=
\sum_{\omega_a \otimes \omega_b \text{ is a summand in } \atop \Delta_{\shuffle}((\omega_{1,1}\cdots \omega_{1,n_1} \circ \tau_1)) \cdots \Delta_{\shuffle}((\omega_{k,1} \cdots \omega_{k,n_k} \circ \tau_k))} \omega_{1,1} \centerdot \dots \centerdot \omega_{k,n_k} \otimes \tau_1\cdots\tau_k \\
&=
\sum_{\omega_a \otimes \omega_b \text{ is a summand in } \atop (\omega_{1,1}\cdots \omega_{1,n_1} \circ \tau_1 \otimes \emptyset + \emptyset \otimes \omega_{1,1}\cdots \omega_{1,n_1} \circ \tau_1) \cdots (\omega_{k,1} \cdots \omega_{k,n_k} \circ \tau_k) \otimes \emptyset + \emptyset +\otimes \omega_{k,1} \cdots \omega_{k,n_k} \circ \tau_k)} \omega_{1,1} \centerdot \dots \centerdot \omega_{k,n_k} \otimes \tau_1\cdots\tau_k \\
&=
\sum_{\omega_a \otimes \omega_b \text{ is a summand in } \atop (\omega_{1,1} \cdots \omega_{1,n} \circ \otimes \omega_{2,1} \cdots \omega_{2,m} \circ)\Delta_{\shuffle}(\omega)} \omega_{1,1} \centerdot \cdots \centerdot \omega_{1,n} \centerdot \omega_{2,m} \centerdot \cdots \centerdot \omega_{2,m} \otimes \omega \\
&=
\sum_{\omega_a \text{ is a summand} \atop \text{in } \omega_{1,1} \cdots \omega_{1,n} \circ \omega_1'} 
\sum_{\omega_b \text{ is a summand} \atop \text{in } \omega_{2,1} \cdots x_{2,m} \circ \omega_2'} \omega_{1,1} \centerdot \cdots \centerdot \omega_{1,n} \centerdot \omega_{2,1} \centerdot \dots \centerdot \omega_{2,m} \otimes \omega_1' \shuffle \omega_2' \\
&=
\Big(\sum_{\omega_a \text{ is a summand} \atop \text{in } \omega_1 \cdots \omega_n \circ \omega' } \omega_1 \centerdot \cdots \centerdot \omega_n \otimes \omega' \Big)
\Big(\sum_{\omega_b \text{ is a summand} \atop  \text{in } \omega_1 \cdots \omega_m \circ \omega'} \omega_1 \centerdot \cdots \centerdot \omega_m \otimes \omega'\Big) \\
&=\rho(\omega_a) \rho(\omega_b),
\end{align*}
where $\tau$'s are trees and we use the fact that the forests on the lefthand side of $\rho$ are primitive to the coshuffle coproduct. \\

The identity
\begin{align*}
\rho(\emptyset)=&1 \otimes \emptyset
\end{align*}
is true by definition. \\
The identity
\begin{align*}
(Id \otimes \delta(\emptyset))\rho=\emptyset \epsilon
\end{align*}
follows from the fact that both sides of the equation only evaluate to something nonzero on scalar multiples of the empty forest. \\
In conclusion, the bialgebras are in cointeraction.
\end{proof}

\begin{corollary}
Let $S(Lie(\mathcal{PT}))^{\ast}$ and $\mathcal{OF}^{\ast}$ denote the linear dual spaces of $S(Lie(\mathcal{PT}))$ and $\mathcal{OF}$, respectively. Define a map $\star: S(Lie(\mathcal{PT}))^{\ast} \otimes \mathcal{OF}^{\ast} \to \mathcal{OF}^{\ast}$ by:
\begin{align*}
(x \star y)(\omega)=(x \otimes y)\rho(\omega).
\end{align*}
Now let $\alpha$ denote a character of $\mathcal{Q}$ and let $a,b \in \mathcal{OF}^{\ast}$ be arbitrary, then:
\begin{align*}
\alpha \star (a \star_{N} b)=(\alpha \star a) \star_{N} (\alpha \star b),
\end{align*}
where $\star_N$ is the convolution product in $\mathcal{H}_N$.
\end{corollary}
\begin{proof}
We have
\begin{align*}
\alpha \star (a \star_N b)=(\alpha \otimes a \otimes b)(Id \otimes \Delta_N)\rho
\end{align*}
and
\begin{align*}
(\alpha \star a) \star_N (\alpha \star b)=(\alpha \otimes a \otimes \alpha \otimes b)(\rho\otimes \rho)\Delta_{N}.
\end{align*}
The statement now follows from Proposition \ref{prop::cointercation} and the fact that $\alpha$ is a character with respect to the symmetric product on $\mathcal{Q}$.
\end{proof}

The following proposition states that the bialgebra $\mathcal{H}$ used for $B$-series substitution can be recovered from the bialgebra $\mathcal{Q}$.

\begin{proposition}
Define a map $\pi: S(Lie(\mathcal{PT}))\to \mathcal{F}$ by $\pi(\tau)$ being the unique non-planar tree $\tau'$ such that $\tau$ is a planar embedding of $\tau'$, whenever $\tau$ is a tree and extended by:
\begin{align*}
\pi(1)=&\emptyset,\\
\pi([\tau_1,\tau_2])=&0, \\
\pi(\tau_1 \centerdot \tau_2)=&\pi(\tau_1)\pi(\tau_2).
\end{align*}
Then $\pi: \mathcal{Q} \to \mathcal{H}$ is a surjective biagebra morphism. Hence $\mathcal{Q}/\ker(\pi) \simeq \mathcal{H}$.
\end{proposition}
\begin{proof}
It is clear that $\pi$ is an algebra morphism by definition, it remains to show that $\pi$ is a coalgebra morphism, i.e. that
\begin{align*}
(\pi \otimes \pi)\Delta_{\mathcal{Q}}=\Delta_{\mathcal{H}}\pi.
\end{align*}
It is furthermore clear that both sides in this equation vanishes on Lie brackets. Suppose that $\tau$ is a tree, then:
\begin{align*}
(\pi \otimes \pi)\Delta_{\mathcal{Q}}(\tau)=&\sum_{\tau \text{ is a summand} \atop \text{in } \tau_1\cdots \tau_n \circ \tau'} \pi(\tau_1)\cdots\pi(\tau_n) \otimes \pi(\tau')
\end{align*}
and
\begin{align*}
\Delta_{\mathcal{H}}(\pi(\tau))=\sum_{\pi(\tau) \text{ is a summand} \atop \text{in } \tau_1 \cdots \tau_n \circ \tau'} \tau_1 \cdots \tau_n \otimes \tau'.
\end{align*}
These have to be equal since you get the expression of $\pi(\tau)$ in terms of the single-vertex tree and pre-Lie grafting from the expression of $\tau$ in terms of the single vertex tree and post-Lie grafting by replacing the post-Lie grafting product by the pre-Lie grafting product in the expression.
\end{proof}

\section{Substitution in LB-series} 
\label{section::5}

Now consider the free $D$-algebra $\mathcal{OF}$. It is graded by the number of vertices in a forest. Denote the completion of $\mathcal{OF}$, with respect to this grading, by $\widetilde{\mathcal{OF}}$. Denote its dual by $\mathcal{OF}^{\ast}$, then there is a bijection $\delta:\widetilde{\mathcal{OF}} \to \mathcal{OF}^{\ast}$ given by the dual basis. An LB-series can then be expressed as
\begin{align*}
LB(a,\alpha)=F_a(\delta^{-1}(\alpha)).
\end{align*}
We shall use this description in the sequel.

\begin{lemma}
Let $\alpha\in Lie(\mathcal{PT})^{\ast}$. The map $B_{\alpha}:\mathcal{OF} \to \widehat{\mathcal{OF}}$ given by $B_{\alpha}(\omega)=\delta^{-1}(\alpha \star \delta(\omega))$ is a D-algebra morphism.
\end{lemma}
\begin{proof}
Recall that $\Delta_N$ is dual to the planar Grossman--Larson product \eqref{planarGLprod}. Hence
\begin{align*}
\delta(\omega_1 \diamond \omega_2)(\omega)=(\delta(\omega_1) \star_{N} \delta(\omega_2))(\omega).
\end{align*}
Then we get:
\begin{align*}
B_{\alpha}(\omega_1 \diamond \omega_2)=&\delta^{-1}(\alpha \star \delta(\omega_1 \diamond \omega_2)) \\
=&\delta^{-1}(\alpha \star (\delta(\omega_1) \star_N \delta(\omega_2))) \\
=&\delta^{-1}( (\alpha \star \delta(\omega_1)) \star_N (\alpha \star \delta(\omega_2)  )  ) \\
=&\delta^{-1}(\alpha \star \delta(\omega_1)  ) \diamond \delta^{-1}(\alpha \star \delta(\omega_2)) \\
=&B_{\alpha}(\omega_1) \diamond B_{\alpha}(\omega_2).
\end{align*}
Using Sweedler's notation, $\rho(\omega)=\omega_{(1)} \otimes \omega_{(2)}$, we have that $B_{\alpha}(\omega_1 \omega_2)$ is the sum over all forests $\omega$ such that $\omega_{(2)}$ contains $\omega_1\omega_2$, multiplied by the corresponding $\alpha(\omega_{(1)})$. However, $\omega_{(1)}$ can be split into a part $\omega^1_{(1)}$ consisting of forests that got contracted into vertices in $\omega_1$, and a part $\omega^2_{(1)}$. Then $\omega_{(1)}=\omega^1_{(1)}\centerdot \omega^2_{(2)}$ and $\alpha(\omega_{(1)})=\alpha(\omega^1_{(1)})\alpha(\omega^2_{(1)})$. Hence
\begin{align*}
B_{\alpha}(\omega_1 \omega_2)=B_{\alpha}(\omega_1)B_{\alpha}(\omega_2).
\end{align*}
This then implies:
\begin{align*}
B_{\alpha}(\omega_1 \graft \omega_2)=B_{\alpha}(\omega_1) \graft B_{\alpha}(\omega_2).
\end{align*}
Lastly, we need to show that $B_{\alpha}$ maps derivations to derivations. This follows if we show that $B_{\alpha}$ is a coshuffle morphism, i.e. we have to show that:
\begin{align*}
(B_{\alpha} \otimes B_{\alpha})\Delta_{\shuffle}=\Delta_{\shuffle}B_{\alpha}.
\end{align*}
Let $\tau$ be a tree, then:
\begin{align*}
B_{\alpha}(\tau)=&B_{\alpha}(B^-(\tau) \graft \bullet) \\
=&B_{\alpha}(B^-(\tau)) \graft B_{\alpha}(\bullet) \\
=&B_{\alpha}(B^-(\tau))\graft \delta^{-1}(\alpha).
\end{align*}
Recall from \cite{Ebrahimi-FardLundervoldMunthe-Kaas2014} the identity:
\begin{align*}
\Delta_{\shuffle}(A \graft B)=(A_{(1)} \graft B_{(1)}) \otimes (A_{(2)} \graft B_{(2)}).
\end{align*}
Furthermore, recall that $\delta^{-1}(\alpha)$ is a Lie polynomial. Hence:
\begin{align*}
\Delta_{\shuffle}(B_{\alpha}(\tau))=&B_{\alpha}(B^-(\tau))_{(1)}\graft \emptyset \otimes B_{\alpha}(B^-(\tau)) \graft \delta^{-1}(\alpha) + B_{\alpha}(B^-(\tau))_{(1)} \graft \delta^{-1}(\alpha) \otimes B_{\alpha}(B^-(\tau))_{(2)} \graft \emptyset \\
=& \emptyset \otimes B_{\alpha}(\tau) + B_{\alpha}(\tau) \otimes \emptyset \\
=&(B_{\alpha} \otimes B_{\alpha})\Delta_{\shuffle}(\tau).
\end{align*}
Now for an arbitrary forest, we get:
\begin{align*}
\Delta_{\shuffle}(B_{\alpha}(\tau_1 \cdots \tau_n))
&= \Delta_{\shuffle}(B_{\alpha}(\tau_1)\cdots B_{\alpha}(\tau_m)) \\
&= \Delta_{\shuffle}(B_{\alpha}(\tau_1)) \cdots \Delta_{\shuffle}(B_{\alpha}(\tau_n)) \\
&= (B_{\alpha} \otimes B_{\alpha})\Delta_{\shuffle}(\tau_1)\cdots\Delta_{\shuffle}(\tau_n) \\
&= (B_{\alpha} \otimes B_{\alpha})\Delta_{\shuffle}(\tau_1 \cdots \tau_n).
\end{align*}
Hence $B_{\alpha}$ is a coshuffle morphism and therefore maps derivations to derivations. In conclusion, $B_{\alpha}$ is a $D$-algebra morphism.
\end{proof}

\begin{remark}
Let $\alpha \in \mathcal{OF}^{\ast}$ be a logarithmic linear map, then it can be described by Lie polynomials in the dual basis, e.g.
\begin{align*}
\alpha=\delta(\Forest{[[]]}\Forest{[]})-\delta(\Forest{[]}\Forest{[[]]})
\end{align*}
is logarithmic. Then by the embedding $Lie(\mathcal{PT})\subset \mathcal{OF}$, we can view $\alpha$ as a linear map $\hat{\alpha}\in Lie(\mathcal{PT})^{\ast}$. For the example $\alpha$ in this remark, we would have:
\begin{align*}
\hat{\alpha}=\delta([\Forest{[[]]},\Forest{[]}]).
\end{align*}
\end{remark}

\begin{theorem}
Let $\alpha,\beta$ be linear maps on $\mathcal{OF}$, with $\alpha$ defining a logarithmic LB-series. Since $\alpha$ is logarithmic, we can view it as a linear map $\hat{\alpha}\in Lie(\mathcal{PT})^{\ast}$. Extend $\hat{\alpha}$ to be a character on $S(Lie(\mathcal{PT}))$, then:
\begin{align*}
LB(LB(a,\alpha),\beta)=LB(a,\hat{\alpha} \star\beta).
\end{align*}
\end{theorem}

\begin{proof}
Denote by $A_{\alpha}: \mathcal{OF} \to \mathcal{OF}^{\ast}$ the unique $D$-algebra morphism given by $A_{\alpha}(\bullet)=\delta^{-1}(\alpha)$. Extend $A_{\alpha}$ to be defined on $\widetilde{\mathcal{OF}}$. Then:
\begin{align*}
F_{LB(a,\alpha)}=F_a \circ A_{\alpha},
\end{align*}
where $\circ$ means composition of functions. So that:
\begin{align*}
LB(LB(a,\alpha),\beta)=F_a(A_{\alpha}(\delta^{-1}(\beta))).
\end{align*}
Furthermore:
\begin{align*}
LB(a,\hat{\alpha} \star \beta)=&F_a(\delta^{-1}(\hat{\alpha} \star \beta))\\
=&F_a(B_{\hat{\alpha}}(\delta^{-1}(\beta))).
\end{align*}
Now the theorem follows if $B_{\hat{\alpha}}=A_{\alpha}$. However, it is clear that $B_{\hat{\alpha}}(\bullet)=A_{\alpha}(\bullet)$. Then equality everywhere follows as both maps are $D$-algebra morphisms.
\end{proof}

It is now worthwhile to relate the previous result, with the recursive substitution formula from \cite{LundervoldMunthe-Kaas2011}.

\begin{remark}
It is proved in \cite{LundervoldMunthe-Kaas2011} that
\begin{align*}
LB(LB(a,\alpha),\beta)=LB(a,A_{\alpha}(\beta)). 
\end{align*}
however, no efficient method for evaluating $A_{\alpha}$ is given. Instead, the dual operator, $A_{\alpha}^\dagger$, is defined by
\begin{align*}
\langle A_{\alpha}(\omega_1),\omega_2 \rangle = \langle \omega_1,A_{\alpha}^\dagger(\omega_2) \rangle.
\end{align*}
Then a recursive formula for evaluating $A_{\alpha}^\dagger$ is shown. Now the observation that $B_{\hat{\alpha}}=A_{\alpha}$ yields:
\begin{align*}
A_{\alpha}^\dagger(\omega)=\hat{\alpha}(\omega_{(1)})\omega_{(2)}.
\end{align*}
In particular, this means that the recursive formula given for $A_{\alpha}^\dagger$ can be used to evaluate $\rho$.
\end{remark}

\begin{proposition}
Let $\hat{\alpha} \in Lie(\mathcal{PT})^{\ast}$, then the map $\hat{\alpha} \star : \mathcal{OF}^{\ast} \to \mathcal{OF}^{\ast}$ is an automorphism over the group of exponential linear maps.
\end{proposition}
\begin{proof}
This follows from $B_{\hat{\alpha}}$ being a coshuffle morphism.
\end{proof}

\section{A coaction of ordered forest contractions} 
\label{section::4}

In this section we shall describe the combinatorial picture of the coaction $\rho$. \\

Recall that the coaction $\rho$ is defined on ordered forests, as a sum over all the ways to obtain a forest $\omega$ by inserting Lie polynomials $\omega_1, \ldots, \omega_n$ into a forest $\omega'$. There is a bijection between vertices in $\omega$ and vertices in $\omega_1, \ldots, \omega_n$. We say that a partition of the vertices of a forest $\omega$ into subforests $\omega_1, \ldots, \omega_n$ is an admissible partition if there exists a forest $\omega'$ and a way to insert Lie brackets into each $\omega_i$ such that $\omega$ is a summand in $\omega_1 \cdots \omega_n \circ \omega'$. If $\omega_1 \cdots \omega_n$ is an admissible partition of $\omega$, we shall denote by $\omega/\omega_1 \cdots \omega_n$ the sum of all $\omega'$ such that $\omega$ is a summand in $\omega_1 \cdots \omega_n \circ \omega'$. We illustrate this with an example.

\begin{example}
The coaction $\rho$ maps the forest $\Forest{[1[3[4]][2]]}$ to:
\begin{align*}
\rho(\Forest{[[[]][]]})
&=\Forest{[[[]][]]} \otimes \bullet + \bullet \centerdot \bullet \centerdot \bullet \centerdot \bullet \otimes \Forest{[[[]][]]} + \bullet \centerdot \Forest{[[[]]]} \otimes \Forest{[[]]} + \bullet \centerdot \Forest{[[][]]} \otimes \Forest{[[]]}+\bullet \centerdot \bullet \centerdot \Forest{[[]]} \otimes 3 \Forest{[[][]]}+[\Forest{[]},\Forest{[[]]}]\centerdot \bullet \otimes \Forest{[[]]},
\end{align*}
corresponding to the following admissible partitions:
\begin{align*}
(\Forest{[1[3[4]][2]]}),(\Forest{[1]},\Forest{[2]},\Forest{[3]},\Forest{[4]}),(\Forest{[2]},\Forest{[1[3[4]]]}),(\Forest{[4]},\Forest{[1[3][2]]}),(\Forest{[2]},\Forest{[4]},\Forest{[1[3]]}),(\Forest{[1]},\Forest{[2]},\Forest{[3[4]]}),(\Forest{[1]},\Forest{[2]}\Forest{[3[4]]}).
\end{align*}
The righthand side of the tensors are the corresponding $\omega'$. \\
Note that the purpose of the labelling of the vertices is to distinguish the different vertices in a partition. It does not relate to any operadic structure, as the coaction is defined over unlabelled forests. \\
Furthermore note that the lefthand side of the tensors is in $S(Lie(\mathcal{PT}))$, while the subforests in the admissible partitions are in $\mathcal{OF}$.
\end{example}

We define a coaction $\Delta_W: \mathcal{OF} \to S(\mathcal{OF^+}) \otimes \mathcal{OF}$ by:
\begin{align*}
\Delta_W(\emptyset)=&1 \otimes \emptyset, \\
\Delta_W(\omega)=&\sum_{\omega_1 \cdots \omega_n \text{ admissible partition}} \omega_1 \centerdot \dots \centerdot \omega_n \otimes \omega/\omega_1\cdots \omega_n,
\end{align*}
where $(S(\mathcal{OF}^+),\centerdot)$ is the symmetric algebra of non-empty ordered forests. We can informally see this as the coaction obtained from $\rho$ by ''removing" the Lie brackets, as shown by the following example.

\begin{example}
\begin{align*}
\Delta_W(\Forest{[[[]][]]}) 
&=\Forest{[[[]][]]} \otimes \bullet + \bullet \centerdot \bullet \centerdot \bullet \centerdot \bullet \otimes \Forest{[[[]][]]} + \bullet \centerdot \Forest{[[[]]]} \otimes \Forest{[[]]} + \bullet \centerdot \Forest{[[][]]} \otimes \Forest{[[]]}+\bullet \centerdot \bullet \centerdot \Forest{[[]]} \otimes 3 \Forest{[[][]]}+\Forest{[]}\Forest{[[]]}\centerdot \bullet \otimes \Forest{[[]]}.
\end{align*}
\end{example}

The reason we may want to consider $\Delta_W$ instead of $\rho$ is because it eliminates the need to rewrite the logarithmic linear map $\alpha \in \mathcal{OF}^{\ast}$ into a map $\hat{\alpha}\in Lie(\mathcal{PT})^{\ast}$, as shown in the following proposition.

\begin{proposition}
Let $\alpha \in \mathcal{OF}^{\ast}$ be logarithmic and let $\hat{\alpha}\in Lie(\mathcal{PT})^{\ast}$ be $\alpha$ restricted to the Lie polynomials. Let $\beta \in \mathcal{OF}^{\ast}$ be arbitrary. Define a map $\star_W : S(\mathcal{OF}^+)^{\ast} \otimes \mathcal{OF}^{\ast} \to \mathcal{OF}^{\ast}$ by:
\begin{align*}
(x \star_W y)=(x \otimes y)\Delta_W,
\end{align*}
Extend $\alpha,\hat{\alpha}$ multiplicatively to $S(\mathcal{OF}^+)^{\ast}$ and $S(Lie(\mathcal{PT}))^{\ast}$, respectively, then:
\begin{align*}
\alpha \star_W \beta = \hat{\alpha} \star \beta.
\end{align*}
\end{proposition}

\begin{proof}
The coactions $\rho$ and $\Delta_W$ agree on every term where the lefthand side consists only of symmetric products of trees. The maps $\alpha$ and $\hat{\alpha}$ agree on trees. Use the Jacobi identity and the anti-symmetry of the Lie bracket to write all (nested) Lie brackets on the lefthand side of $\rho(\omega)$ in such a way that the left-to-right order of the trees in the Lie bracket agrees with the left-to-right order of the trees seen as subtrees in the planar embedding of $\omega$. If the lefthand side of the coaction $\rho$ contains a term with a Lie bracket $[\omega_1,\omega_2]$, then the corresponding term in $\Delta_W$ contains instead a term $\omega_1\omega_2$ (iterate for nested brackets). If $\hat{\alpha}$ evaluates to $c\in \mathbb{K}$ on $[\omega_1,\omega_2]$, then $\alpha$ evaluates to $c$ on $\omega_1\omega_2$ and to $-c$ on $\omega_2\omega_1$.  Every term in $\rho$ and $\Delta_W$ agrees on the right side.
\end{proof}

\begin{corollary}
Let $\alpha \in \mathcal{OF}^{\ast}$ be logarithmic. Extend $\alpha$ multiplicatively to $S(\mathcal{OF}^+)^{\ast}$, then:
\begin{align*}
LB(LB(a,\alpha),\beta)=LB(a,\alpha \star_W \beta).
\end{align*}
Furthermore, the map $\alpha \star_W : \mathcal{OF}^{\ast} \to \mathcal{OF}^{\ast}$ is an automorphism over the group of exponential linear maps.
\end{corollary}

We are now ready to formulate a combinatorial description of admissible partitions.

\begin{proposition} \label{prop::8}
Let $\omega$ be a forest and let $\omega_1 \cdots \omega_n$ be a partition of the vertices of $\omega$ into subforests. This partition is admissible if and only if the following conditions are met:
\begin{enumerate}
\item Each root in the same $\omega_i$ are either roots of $\omega$ or grafted onto the same vertex of $\omega$. Furthermore, the roots of $\omega_i$ are adjacent in the planar embedding of $\omega$.
\item If $e$ is an edge in an $\omega_i$ that goes between different vertices in $\omega_i$, then every edge $e'$ in $\omega$ that is outgoing from the same vertex as $e$ and is to the right of $e$ in the planar embedding, is also an edge between vertex of $\omega_i$.
\end{enumerate}
\end{proposition}

\begin{proof}
First suppose that the partition is admissible. Then there exists some forest $\omega'$ and some way to insert Lie brackets into each $\omega_i$ such that $\omega$ is a summand in $\omega_1 \cdots \omega_n \circ \omega'$. If a vertex $i$ is grafted on a vertex $j$ in $\omega'$, then this amounts to a Lie polynomial $\omega_i$ being grafted onto a Lie polynomial $\omega_j$ in $\omega_1 \cdots \omega_n \circ \omega'$. However, since $\omega_i$ is a Lie polynomial, the terms where different roots of $\omega_i$ goes onto different vertices of $\omega_j$ will cancel. Hence condition $1.$ is satisfied. Furthermore, the edges going from vertices in $\omega_j$ to vertices in $\omega_i$ must be to the left of edges going between vertices in $\omega_j$, since all grafting is done in the leftmost position. Hence condition $2$ is satisfied. \\
Now suppose that $\omega_1\cdots \omega_n$ is a partition that satisfies conditions $1.$ and $2.$ of the proposition. Let $\omega'$ denote the forest on $n$ vertices obtained by adding an edge from vertex $i$ to vertex $j$ if there is an edge from $\omega_i$ to $\omega_j$ in $\omega$. Condition $1.$ ensures that this is unambiguous. Now endow $\omega'$ with a planar embedding such that vertex $i$ is to the left of vertex $j$ if $\omega_i$ is to the left of $\omega_j$ in $\omega$. The choice of planar embedding is not unique. Now turn each $\omega_i$ into a Lie polynomial by insertion of Lie brackets and consider $\omega_1 \cdots \omega_n \circ \omega'$. This results in a sum over all ways to graft $\omega_i$ onto $\omega_j$ if there was an edge from $\omega_i$ to $\omega_j$ in $\omega$. Because of condition $2.$, the original placements of the edges from $\omega$ will appear in this sum.
\end{proof}

\begin{proposition}
Let $\omega$ be a forest and let $\omega_1 \cdots \omega_n$ be an admissible partition. Let $\omega'$ denote the forest on $n$ vertices obtained by adding an edge from vertex $i$ to vertex $j$ if there is an edge from $\omega_i$ to $\omega_j$ in $\omega$. Then $\omega/\omega_1\cdots\omega_n$ is the sum over all ways to endow the forest $\omega'$ with a planar embedding such that vertex $i$ is to the left of vertex $j$ if $\omega_i$ is to the left of $\omega_j$.
\end{proposition}

\begin{proof}
It was shown in the proof of proposion \ref{prop::8} that each such embedding is a summand in $\omega/\omega_1 \cdots \omega_n$. Now suppose there exists an $\omega''$ such that $\omega$ is a summand in $\omega_1 \cdots \omega_n \circ \omega''$ but $\omega''$ is not such a planar embedding of $\omega'$. If $\omega''$ is a planar embedding of $\omega'$ with vertex $i$ to the right of vertex $j$ but $\omega_i$ is to the left of $\omega_j$, then it is clear that $\omega_1 \cdots \omega_n \circ \omega''$ cannot produce the planar embedding of $\omega$. If $\omega''$ does not have an edge from vertex $i$ to vertx $j$ but $\omega$ has an edge from $\omega_i$ to $\omega_j$, then it is also clear that $\omega_1 \cdots \omega_n \circ \omega''$ cannot produce $\omega$.
\end{proof}

\begin{proposition}
Extend $\Delta_W$ to be defined on symmetric products of forests, $\Delta_{W}: S(\mathcal{OF^+}) \to S(\mathcal{OF^+}) \otimes S(\mathcal{OF^+})$ by
\begin{align*}
\Delta_W(\omega_1 \centerdot \omega_2)=\Delta_W(\omega_1) \centerdot \Delta_W(\omega_2).
\end{align*}
Then $(S(\mathcal{OF}),\Delta_W,\epsilon)$ is a coalgebra.
\end{proposition}
\begin{proof}
The only non-trivial thing to show is:
\begin{align*}
(Id \otimes \Delta_W)\Delta_W=(\Delta_W \otimes Id)\Delta_W.
\end{align*}
This however follows from the fact that if $\omega_1, \dots, \omega_n$ is an admissible partition of $\omega$ and $\omega_i^1,\dots, \omega_i^{n_i}$ is an admissible partition of $\omega_i$, for $i=1,\dots,n$, then $\omega_1^1,\dots,\omega_1^{n_1},\omega_2^2,\dots,\omega_n^{n_n}$ is an admissible partition of $\omega$.
\end{proof}

\vspace{1cm}

{\bf{Acknowledgement:}}
The author is supported by the Research Council of Norway through project 302831 ''Computational Dynamics and Stochastics on Manifolds" (CODYSMA). This work was partially supported by the project Pure Mathematics in Norway, funded by Trond Mohn Foundation and Troms{\o} Research Foundation. \\
The author thanks Kurusch Ebrahimi-Fard and Hans-Munthe Kaas for the helpful discussions. He furthermore thanks Dominique Manchon for reading the paper and for his suggestions.

\bibliographystyle{acm}
\bibliography{LieButcherReferences}
\end{document}